\documentclass{article}%
\usepackage{amssymb}
\usepackage{graphicx}
\usepackage{amsmath}
\usepackage{amsfonts}%
\providecommand{\U}[1]{\protect\rule{.1in}{.1in}}
\newtheorem{theorem}{Theorem}

\newtheorem{axiom}[theorem]{Axiom}

\newtheorem{corollary}[theorem]{Corollary}

\newtheorem{definition}[theorem]{Definition}

\newtheorem{lemma}[theorem]{Lemma}
\newtheorem{notation}[theorem]{Notation}

\newtheorem{proposition}[theorem]{Proposition}
\newtheorem{remark}[theorem]{Remark}

\newenvironment{proof}[1][Proof]{\textbf{#1.} }{\ \rule{0.5em}{0.5em}}
\begin{document}

\title{Synthetic Differential Geometry \\within \\Homotopy Type Theory\\I}
\author{Hirokazu NISHIMURA\\Institute of Mathematics, University of Tsukuba\\Tsukuba, Ibaraki 305-8571\\Japan}
\maketitle

\begin{abstract}
Both syntheticc differential geometry and homotopy type theory prefer
synthetic arguments to analytical ones. This paper gives a first step towards
developing synthetic differential geometry within homotopy type theory. Model
theory of this approach will be discussed in a subsequent paper.

\end{abstract}

\section{\label{s0}Introduction}

Homotopy type theory (cf. \cite{vo}), born at the crossroads of type theory
and homotopy theory in the first decade of this century (\cite{aw} and
\cite{vo1}) inspired by \cite{hof}, is expected to give a solid foundation to
mathematics. A large portion of classical homotopy theory has already been
developed within homotopy type theory with new formulations and new proofs of
celebrated classical results such as the Freudenthal suspension theorem, the
van Kampen theorem and the Whitehead theorem being discovered by the intimate
collaboration of men and the proof assistant system COQ in the process of developing.

Synthetic differential geometry is developed synthetically by using nilpotent
infinitesimals. For standard textbooks on synthetic differential geometry the
reader is referred to \cite{ko3} and \cite{la}. The principal objective in
this paper is to develop synthetic differential geometry within homotopy type
theory. Since both theories prefer synthetic arguments to analytic ones, there
is a tremendous affinity between them. In the next section (\S \ref{s1}) we
will set up the foundation for types of nilpotent infinitesimals and announce
the homotopical generalized Kock-Lawvere axiom. After enjoying elementary
differential calculus up to the Taylor expansion (cf. \cite{ko1} and
\cite{ko2}) in \S \ref{s2}, we will discuss microlinearity in \S \ref{s3} and
tangency in in \S \ref{s4} by using the machinery of set truncation.
\S \ref{s5} is devoted to strong differences. It culminates in a streamlined
presentation of the general Jacobi identity discussed in \cite{ni1} and
\cite{ni2}. The last section (\S \ref{s6}) deals with vector fields on a
microlinear type. In a subsequent paper we will discuss model theory of this approach.

\section{\label{s1}Nilpotent Infinitesimals}

\begin{axiom}
\label{a1.0}The type $\mathbb{R}$ is a set which is a $\mathbb{Q}$-algebra,
where $\mathbb{Q}$\ is the type of rational numbers.
\end{axiom}

\begin{definition}
A finitely presented $\mathbb{R}$-algebra of the form%
\[
\mathbb{R}\left[  X_{1},...,X_{n}\right]  /\left(  X_{1}^{m_{1}}%
,...,X_{n}^{m_{n}},f_{1}\left(  X_{1},...,X_{n}\right)  ,...,f_{k}\left(
X_{1},...,X_{n}\right)  \right)
\]
with $f_{i}$'s\ being polynomials in $X_{1},...,X_{n}$\ with coefficients in
$\mathbb{R}$\ is called a \underline{Weil algebra}. It should be recalled that
finitely presented $\mathbb{R}$-algebras are to be defined by higher induction
in homotopy type theory.
\end{definition}

\begin{notation}
Given a Weil algebra $\mathfrak{W}$, we denote by $\mathrm{Spec}_{\mathbb{R}%
}\mathfrak{W}$\ the type of homomorphisms of $\mathbb{R}$-algebras from the
$\mathbb{R}$-algebra $\mathfrak{W}$\ to the $\mathbb{R}$-algebra $\mathbb{R}$
By way of example, the type%
\[
\mathrm{Spec}_{\mathbb{R}}\mathbb{R}\left[  X\right]  /(X^{2})
\]
is equivalent to the subtype%
\[
D:\equiv\left\{  d:\mathbb{R}\mid d^{2}=0\right\}
\]
of the type $\mathbb{R}$, while the type%
\[
\mathrm{Spec}_{\mathbb{R}}\mathbb{R}\left[  X,Y\right]  /(X^{2},Y^{2},XY)
\]
is equivalent to the subtype%
\[
D\left(  2\right)  :\equiv\left\{  \left(  d_{1},d_{2}\right)  :D^{2}\mid
d_{1}d_{2}=0\right\}
\]
of the type $D^{2}$.
\end{notation}

\begin{definition}
Given a Weil algebra $\mathfrak{W}$, the type $\mathrm{Spec}_{\mathbb{R}%
}\mathfrak{W}$\ is called the \underline{infinitesimal type} associated to the
Weil algebra $\mathfrak{W}$.
\end{definition}

\begin{definition}
The diagram of infinitesimal types resulting from a finite limit diagram of
Weil algebras by application of the contravariant functor $\mathrm{Spec}%
_{\mathbb{R}}$\ is called a \underline{quasi-colimit diagram of infinitesimal
types}.
\end{definition}

\begin{axiom}
\label{a1.1}(\underline{Homotopical Generalized Kock-Lawvere Axiom}) Given a
Weil algebra $\mathfrak{W}$, the canonical homomorphism of $\mathbb{R}%
$-algebras\ from the $\mathbb{R}$-algebra $\mathfrak{W}$ to the $\mathbb{R}%
$-algebra $\mathrm{Spec}_{\mathbb{R}}\mathfrak{W}\rightarrow\mathbb{R}$\
\[
\lambda_{x:\mathfrak{W}}\lambda_{f:\mathrm{Spec}_{\mathbb{R}}\mathfrak{W}%
}f\left(  x\right)
\]
is an equivalence, namely,
\[
\mathfrak{W}\backsimeq\mathrm{Spec}_{\mathbb{R}}\mathfrak{W}\rightarrow
\mathbb{R}%
\]

\end{axiom}

\begin{remark}
Under Axiom \ref{a1.1}, a finite diagram of infinitesimal types is a
quasi-colimit diagram iff the diagram resulting from it by application of the
contravariant functor $\rightarrow\mathbb{R}$\ is a limit diagram.
\end{remark}

We recall the notion of a simplicial small object introduced in \S 4 of
\cite{ni1}.

\begin{notation}
(\underline{Simplicial infinitesimal types}) Given $n:\mathbb{N}$ and a finite
set $\mathfrak{p}$\ of lists of natural numbers $i$ with $1\leq i\leq n$, we
denote by $D^{n}\left\{  \mathfrak{p}\right\}  $ a set
\[
\left\{  \left(  d_{1},...,d_{n}\right)  :D^{n}\mid\prod_{\left(
i_{1},...,i_{k}\right)  :\mathfrak{p}}d_{i_{1}}...d_{i_{k}}=0\right\}
\]
By way of example, we have
\begin{align*}
D\left(  2\right)   &  =D^{2}\left\{  \left(  1,2\right)  \right\} \\
D\left(  3\right)   &  =D^{3}\left\{  \left(  1,2\right)  ,\left(  1,3\right)
,\left(  2,3\right)  \right\}
\end{align*}
while both $D^{2}\left\{  \left(  1\right)  \right\}  $\ and $D^{2}\left\{
\left(  2\right)  \right\}  $\ are equivalent to $D$\ via the equivalences
$\lambda_{d:D}\left(  0,d\right)  $ and $\lambda_{d:D}\left(  d,0\right)  $\ respectively.
\end{notation}

\begin{axiom}
\label{a2}The type $\mathbb{R}$ is a set endowed with a structure of a unitary
commutative ring such that
\begin{equation}
\prod_{f::D\rightarrow\mathbb{R}}\mathrm{isContr}\left(  \sum_{a:\mathbb{R}%
}\prod_{d:D}f\left(  d\right)  =_{\mathbb{R}}f\left(  0\right)  +ad\right)
\label{a2.1}%
\end{equation}
where $D$\ stands for the subtype
\[
\sum_{x:\mathbb{R}}x^{2}=_{\mathbb{R}}0
\]
of $\mathbb{R}$.
\end{axiom}

\section{\label{s2}Elementary Differential Calculus}

\begin{notation}
Given $f:\mathbb{R}\rightarrow\mathbb{R}$ and $x:\mathbb{R}$, we write
\[
f^{\prime}\left(  x\right)  :\mathbb{R}%
\]
for one of the propositionally identical $a:\mathbb{R}$ abiding by
\[
\prod_{d:D}\left(  \lambda_{y:D}f\left(  x+y\right)  \right)  \left(
d\right)  =_{\mathbb{R}}\left(  \lambda_{y:D}f\left(  x+y\right)  \right)
\left(  0\right)  +ad
\]
\ 
\end{notation}

\begin{proposition}
\label{p2.1}We have
\[
\prod_{f:\mathbb{R}\rightarrow\mathbb{R}}\mathrm{isContr}\left(
\sum_{g:\mathbb{R}\rightarrow\mathbb{R}}\prod_{x:\mathbb{R}}\prod
_{d:D}f\left(  x+d\right)  =_{\mathbb{R}}f\left(  x\right)  +g\left(
x\right)  d\right)
\]

\end{proposition}

\begin{proof}
This follows from the above axiom by the principle of unique choice (\S 3.9 of
\cite{vo}).
\end{proof}

\begin{notation}
Given $f:\mathbb{R}\rightarrow\mathbb{R}$, we write
\[
f^{\prime}:\mathbb{R}\rightarrow\mathbb{R}%
\]
for one of the propositionally identical $g:\mathbb{R}\rightarrow\mathbb{R}$
abiding by
\[
\prod_{x:\mathbb{R}}\prod_{d:D}f\left(  x+d\right)  =_{\mathbb{R}}f\left(
x\right)  +g\left(  x\right)  d
\]
Given $n:\mathbb{N}$, we can define
\[
f^{\left(  n\right)  }:\mathbb{R}\rightarrow\mathbb{R}%
\]
inductively on $n$.
\end{notation}

\begin{proposition}
\label{p2.2}We have
\[
\prod_{f,g::\mathbb{R}\rightarrow\mathbb{R}}\prod_{x:\mathbb{R}}\left(
fg\right)  ^{\prime}\left(  x\right)  =f^{\prime}\left(  x\right)  g\left(
x\right)  +f\left(  x\right)  g^{\prime}\left(  x\right)
\]

\end{proposition}

\begin{proof}
Let $d:D$. We have
\begin{align*}
f(x+d)g(x+d)  &  =\left(  f(x)+f^{\prime}(x)d\right)  \left(  g(x)+g^{\prime
}(x)d\right) \\
&  =f(x)g(x)+\left(  f^{\prime}\left(  x\right)  g\left(  x\right)  +f\left(
x\right)  g^{\prime}\left(  x\right)  \right)  d+f^{\prime}(x)g^{\prime
}(x)d^{2}\\
&  =f(x)g(x)+\left(  f^{\prime}\left(  x\right)  g\left(  x\right)  +f\left(
x\right)  g^{\prime}\left(  x\right)  \right)  d\\
&  \text{[since }d^{2}\text{ vanishes]}%
\end{align*}
so that the desired conclusion follows.
\end{proof}

\begin{proposition}
\label{p2.3}We have
\[
\prod_{f,g::\mathbb{R}\rightarrow\mathbb{R}}\prod_{x:\mathbb{R}}\left(  g\circ
f\right)  ^{\prime}\left(  x\right)  =g^{\prime}\left(  f\left(  x\right)
\right)  f^{\prime}\left(  x\right)
\]

\end{proposition}

\begin{proof}
Let $d:D$. We have
\begin{align*}
g(f(x+d))  &  =g\left(  f(x)+f^{\prime}(x)d\right) \\
&  =g\left(  f(x)\right)  +g^{\prime}\left(  f\left(  x\right)  \right)
\left(  f^{\prime}\left(  x\right)  d\right) \\
&  \text{[since }f^{\prime}(x)d:D\text{]}\\
&  =g\left(  f(x)\right)  +\left(  g^{\prime}\left(  f\left(  x\right)
\right)  f^{\prime}\left(  x\right)  \right)  d
\end{align*}
so that the desired conclusion follows.
\end{proof}

\begin{notation}
Given $n:\mathbb{N}$, we write
\[
\mathrm{List}_{D}\left(  n\right)
\]
for the type of lists of elements in $D$\ with length $n$. Thus the type
$\mathrm{List}_{D}\left(  n\right)  $\ consists of $\left(  d_{1}%
,...,d_{n}\right)  $'s with $d_{i}:D$ $\left(  1\leq i\leq n\right)  $. In
particular, $\mathrm{List}_{D}\left(  0\right)  $\ consists only of $\left(
{}\right)  $. Given $m,n:\mathbb{N}$,\ we define
\[
\mathrm{Sym}_{n,m}:\mathrm{List}_{D}\left(  n\right)  \rightarrow\mathbb{R}%
\]
by induction on $n$. We decree that
\[
\mathrm{Sym}_{0,0}:\equiv\lambda_{x:\mathrm{List}_{D}\left(  0\right)  }1
\]
and
\[
\mathrm{Sym}_{0,m+1}:\equiv\lambda_{x:\mathrm{List}_{D}\left(  0\right)  }0
\]
whatever $m$\ may be. We decree that
\[
\mathrm{Sym}_{n+1,0}:\equiv\lambda_{x:\mathrm{List}_{D}\left(  n+1\right)  }1
\]
and that
\begin{align*}
&  \mathrm{Sym}_{n+1,m+1}\left(  d_{1},...,d_{n+1}\right) \\
&  :\equiv\mathrm{Sym}_{n,m+1}\left(  d_{1},...,d_{n}\right)  +d_{n+1}%
\mathrm{Sym}_{n,m}\left(  d_{1},...,d_{n}\right)
\end{align*}
whatever $m$\ may be.
\end{notation}

It is easy to see that

\begin{lemma}
Given $m,n:\mathbb{N}$, we have
\[
\mathrm{Sym}_{n,m}=\lambda_{x:\mathrm{List}_{D}\left(  n\right)  }0
\]
provided that $n<m$.
\end{lemma}

Now we have the infinitesimal Taylor expansion theorem.

\begin{theorem}
\label{t2.1}Given $f:\mathbb{R}\rightarrow\mathbb{R}$,\ $x:\mathbb{R}$\ and
$n:\mathbb{N}$, we have
\begin{align*}
&  f\left(  x+\mathrm{Sym}_{n,1}\left(  d_{1},...,d_{n}\right)  \right) \\
&  =f\left(  x\right)  +f^{\prime}\left(  x\right)  \mathrm{Sym}_{n,1}\left(
d_{1},...,d_{n}\right)  +f^{\prime\prime}\left(  x\right)  \mathrm{Sym}%
_{n,2}\left(  d_{1},...,d_{n}\right)  +...\\
&  +f^{\left(  i\right)  }\mathrm{Sym}_{n,i}\left(  d_{1},...,d_{n}\right)
+...+f^{\left(  n\right)  }\mathrm{Sym}_{n,n}\left(  d_{1},...,d_{n}\right)
\end{align*}

\end{theorem}

\begin{proof}
By induction on $n$. If $n=0$, the theorem holds trivially. We have
\begin{align*}
&  f\left(  x+\mathrm{Sym}_{n+1,1}\left(  d_{1},...,d_{n+1}\right)  \right) \\
&  =f\left(  x+\mathrm{Sym}_{n,1}\left(  d_{1},...,d_{n}\right)
+d_{n+1}\mathrm{Sym}_{n,0}\left(  d_{1},...,d_{n}\right)  \right) \\
&  =f\left(  x+\mathrm{Sym}_{n,1}\left(  d_{1},...,d_{n}\right)
+d_{n+1}\right) \\
&  =f\left(  x+\mathrm{Sym}_{n,1}\left(  d_{1},...,d_{n}\right)  \right)
+d_{n+1}f^{\prime}\left(  x+\mathrm{Sym}_{n,1}\left(  d_{1},...,d_{n}\right)
\right) \\
&  =f\left(  x\right)  +f^{\prime}\left(  x\right)  \mathrm{Sym}_{n,1}\left(
d_{1},...,d_{n}\right)  +...+f^{\left(  n\right)  }\left(  x\right)
\mathrm{Sym}_{n,n}\left(  d_{1},...,d_{n}\right)  +\\
&  d_{n+1}\left\{  f^{\prime}\left(  x\right)  +f^{\prime\prime}\left(
x\right)  \mathrm{Sym}_{n,1}\left(  d_{1},...,d_{n}\right)  +...+f^{\left(
n+1\right)  }\left(  x\right)  \mathrm{Sym}_{n,n}\left(  d_{1},...,d_{n}%
\right)  \right\} \\
&  =f\left(  x\right)  +f^{\prime}\left(  x\right)  \left(  \mathrm{Sym}%
_{n,1}\left(  d_{1},...,d_{n}\right)  +d_{n+1}\right)  +\\
&  f^{\prime\prime}\left(  x\right)  \left(  \mathrm{Sym}_{n,2}\left(
d_{1},...,d_{n}\right)  +d_{n+1}\mathrm{Sym}_{n,1}\left(  d_{1},...,d_{n}%
\right)  \right)  +...+\\
&  f^{\left(  n+1\right)  }\left(  x\right)  d_{n+1}\mathrm{Sym}_{n,n}\left(
d_{1},...,d_{n}\right) \\
&  =f\left(  x\right)  +f^{\prime}\left(  x\right)  \mathrm{Sym}%
_{n+1,1}\left(  d_{1},...,d_{n+1}\right)  +f^{\prime\prime}\left(  x\right)
\mathrm{Sym}_{n+1,2}\left(  d_{1},...,d_{n+1}\right)  +...+\\
&  f^{\left(  n+1\right)  }\left(  x\right)  \mathrm{Sym}_{n+1,n+1}\left(
d_{1},...,d_{n+1}\right)
\end{align*}

\end{proof}

The familiar form of the Taylor expansion theorem goes as follows:

\begin{corollary}
We assume that the ring $\mathbb{R}$\ is an algebra over the rationals
$\mathbb{Q}$. Given $f:\mathbb{R}\rightarrow\mathbb{R}$,\ $x:\mathbb{R}$\ and
$n:\mathbb{N}$, we have
\begin{align*}
&  f\left(  x+\mathrm{Sym}_{n,1}\left(  d_{1},...,d_{n}\right)  \right) \\
&  =f\left(  x\right)  +f^{\prime}\left(  x\right)  \mathrm{Sym}_{n,1}\left(
d_{1},...,d_{n}\right)  +\frac{1}{2}f^{\prime\prime}\left(  x\right)  \left(
\mathrm{Sym}_{n,1}\left(  d_{1},...,d_{n}\right)  \right)  ^{2}+...+\\
&  \frac{1}{i!}\left(  \mathrm{Sym}_{n,1}\left(  d_{1},...,d_{n}\right)
\right)  ^{i}+...+\frac{1}{n!}\left(  \mathrm{Sym}_{n,1}\left(  d_{1}%
,...,d_{n}\right)  \right)  ^{n}%
\end{align*}

\end{corollary}

\begin{proof}
This follows directly from the theorem simply by observing that
\[
i!\mathrm{Sym}_{n,i}\left(  d_{1},...,d_{n}\right)  =\left(  \mathrm{Sym}%
_{n,1}\left(  d_{1},...,d_{n}\right)  \right)  ^{i}\quad\left(  1\leq i\leq
m\right)
\]

\end{proof}

\begin{definition}
An $\mathbb{R}$-module $E$\ is called \underline{Euclidean} if it abides by
the following condition:
\[
\prod_{f::D\rightarrow E}\mathrm{isContr}\left(  \sum_{a:E}\prod_{d:D}f\left(
d\right)  =f\left(  0\right)  +ad\right)
\]

\end{definition}

Given $X:\mathcal{U}$ and an $\mathbb{R}$-module $E\left(  x\right)  $ for
each $x:X$, the type $\prod_{x:X}E\left(  x\right)  $ is naturally an
$\mathbb{R}$-module. It is easy to see that

\begin{proposition}
\label{p2.4}If the $\mathbb{R}$-module $E\left(  x\right)  $ is Euclidean for
each $x:X$, then the $\mathbb{R}$-module $\prod_{x:X}E\left(  x\right)  $\ is
also Euclidean.
\end{proposition}

\begin{proof}
By the function extensionality axiom (Axiom 2.9.3 of \cite{vo})\ and the
principle of unique choice (\S 3.9 of \cite{vo}).
\end{proof}

\begin{notation}
Given an $\mathbb{R}$-module $E$, a Euclidean $\mathbb{R}$-module $F$ and
$f:E\rightarrow F$, we write
\[
f^{\prime}:E\rightarrow E\rightarrow F
\]
for one of the propositionally identical $f^{\prime}$ abiding by
\[
\prod_{x:E}\prod_{a:E}\prod_{d:D}f(x+ad)=f(x)+f^{\prime}(x,a)d
\]

\end{notation}

\begin{proposition}
\label{p2.5}Given an $\mathbb{R}$-module $E$, a Euclidean $\mathbb{R}$-module
$F$\ and $f:E\rightarrow F$, we have
\[
\prod_{x:E}\prod_{a:E}\prod_{b:E}f^{\prime}\left(  x,a+b\right)  =f^{\prime
}\left(  x,a\right)  +f^{\prime}\left(  x,b\right)
\]
and
\[
\prod_{x:E}\prod_{a:E}\prod_{r:\mathbb{R}}f^{\prime}\left(  x,ra\right)
=rf^{\prime}\left(  x,a\right)
\]
In other words,
\[
f^{\prime}\left(  x\right)  :E\rightarrow F
\]
is a homomorphism of $\mathbb{R}$-modules.
\end{proposition}

\begin{proof}
Given $d:D$, we have
\begin{align*}
f\left(  x+\left(  a+b\right)  d\right)   &  =f\left(  \left(  x+ad\right)
+bd\right) \\
&  =f\left(  x+ad\right)  +f^{\prime}\left(  x+ad,b\right)  d\\
&  =f\left(  x\right)  +f^{\prime}\left(  x,a\right)  d+\left\{  f^{\prime
}\left(  x,b\right)  +\left(  \lambda_{y:E}f^{\prime}\left(  y,b\right)
\right)  ^{\prime}\left(  x,a\right)  d\right\}  d\\
&  =f\left(  x\right)  +\left(  f^{\prime}\left(  x,a\right)  +f^{\prime
}\left(  x,b\right)  \right)  d+\left(  \lambda_{y:E}f^{\prime}\left(
y,b\right)  \right)  ^{\prime}\left(  x,a\right)  d^{2}\\
&  =f\left(  x\right)  +\left(  f^{\prime}\left(  x,a\right)  +f^{\prime
}\left(  x,b\right)  \right)  d\\
&  \text{[since }d^{2}\text{ vanishes]}%
\end{align*}
while we have
\begin{align*}
f\left(  x+\left(  ra\right)  d\right)   &  =f\left(  x+a\left(  rd\right)
\right) \\
&  =f(x)+f^{\prime}\left(  x,a\right)  \left(  rd\right) \\
&  \text{[since }rd:D\text{]}\\
&  =f(x)+\left(  rf^{\prime}\left(  x,a\right)  \right)  d
\end{align*}
so that the desired conclusion follows.
\end{proof}

\begin{notation}
Given an $\mathbb{R}$-module $E$, a Euclidean $\mathbb{R}$-module $F$\ and
$f:E\rightarrow F$, we have
\[
f^{\prime}:E\rightarrow E\rightarrow F
\]
Since the $\mathbb{R}$-module $E\rightarrow F$\ is Euclidean by Proposition ,
we have
\[
\left(  f^{\prime}\right)  ^{\prime}:E\rightarrow E\rightarrow E\rightarrow F
\]
We will often write $f^{\prime\prime}$ in place of $\left(  f^{\prime}\right)
^{\prime}$.
\end{notation}

It is easy to see that

\begin{proposition}
\label{p2.6}Given an $\mathbb{R}$-module $E$, a Euclidean $\mathbb{R}$-module
$F$\ and $f:E\rightarrow F$, we have
\[
\prod_{x:E}\prod_{a_{1},a_{2}:E}\prod_{b:E}f^{\prime\prime}\left(
x,a_{1}+a_{2},b\right)  =f^{\prime\prime}\left(  x,a_{1},b\right)
+f^{\prime\prime}\left(  x,a_{2},b\right)
\]%
\[
\prod_{x:E}\prod_{a:E}\prod_{b_{1},b_{2}:E}f^{\prime\prime}\left(
x,a,b_{1}+b_{2}\right)  =f^{\prime\prime}\left(  x,a,b_{1}\right)
+f^{\prime\prime}\left(  x,a,b_{2}\right)
\]%
\[
\prod_{x:E}\prod_{a:E}\prod_{b:E}\prod_{r:\mathbb{R}}f^{\prime\prime}\left(
x,ra,b\right)  =rf^{\prime\prime}\left(  x,a,b\right)
\]
and
\[
\prod_{x:E}\prod_{a:E}\prod_{b:E}\prod_{r:\mathbb{R}}f^{\prime\prime}\left(
x,a,rb\right)  =rf^{\prime\prime}\left(  x,a,b\right)
\]
In short, $f^{\prime\prime}\left(  x\right)  $\ is bilinear.
\end{proposition}

\begin{proof}
By Proposition \ref{p2.5}.
\end{proof}

We can say more.

\begin{proposition}
Given an $\mathbb{R}$-module $E$, a Euclidean $\mathbb{R}$-module $F$\ and
$f:E\rightarrow F$, we have
\[
\prod_{x:E}\prod_{a:E}\prod_{b:E}f^{\prime\prime}\left(  x,a,b\right)
=f^{\prime\prime}\left(  x,b,a\right)
\]

\end{proposition}

\begin{proof}
Given $d_{1},d_{2}:D$, we compute
\begin{align*}
&  f(x+ad_{1}+bd_{2})-f(x+ad_{1})-f(x+bd_{2})+f(x)\\
&  =f(x+ad_{1}+bd_{2})-f(x+bd_{2})-f(x+ad_{1})+f(x)
\end{align*}
in two different ways. On the one hand, we have
\begin{align*}
&  f(x+ad_{1}+bd_{2})-f(x+ad_{1})-f(x+bd_{2})+f(x)\\
&  =\left(  f(x+ad_{1}+bd_{2})-f(x+ad_{1})\right)  -\left(  f(x+bd_{2}%
)-f(x)\right) \\
&  =f^{\prime}\left(  x+ad_{1},b\right)  d_{2}-f^{\prime}\left(  x,b\right)
d_{2}\\
&  =\left(  f^{\prime}\left(  x+ad_{1}\right)  -f^{\prime}\left(  x\right)
\right)  \left(  b\right)  d_{2}\\
&  =\left(  f^{\prime\prime}\left(  x,a\right)  d_{1}\right)  \left(
b\right)  d_{2}\\
&  =f^{\prime\prime}\left(  x,a,b\right)  d_{1}d_{2}%
\end{align*}
On the other hand, we have
\begin{align*}
&  f(x+ad_{1}+bd_{2})-f(x+bd_{2})-f(x+ad_{1})+f(x)\\
&  =\left(  f(x+ad_{1}+bd_{2})-f(x+bd_{2})\right)  -\left(  f(x+ad_{1}%
)-f(x)\right) \\
&  =f(x+bd_{2},a)d_{1}-f(x,a)d_{1}\\
&  =\left(  f^{\prime}\left(  x+bd_{2}\right)  -f^{\prime}\left(  x\right)
\right)  \left(  a\right)  d_{1}\\
&  =\left(  f^{\prime\prime}\left(  x,b\right)  d_{2}\right)  \left(
a\right)  d_{1}\\
&  =f^{\prime\prime}\left(  x,b,a\right)  d_{1}d_{2}%
\end{align*}
Therefore the desired conclusion follows.
\end{proof}

\section{\label{s3}Microlinearity}

\begin{definition}
The diagram of small objects resulting from a limit diagram of Weil algebras
by application of the contravariant functor%
\[
\mathrm{Spec}_{\mathbb{R}}%
\]
is called a \underline{quasi-colimit diagram of small objects}. Therefore, by
Axiom \ref{a3}, a diagram $\mathcal{D}$\ of small objects is a quasi-colimit
diagram iff the exponentiation $\mathcal{D}\rightarrow\mathbb{R}$ of the
diagram $\mathcal{D}$\ over the type $\mathbb{R}$\ is a limit diagram.
\end{definition}

\begin{definition}
A type $M$\ is called \underline{microlinear} provided that the exponentiation
$\mathcal{D}\rightarrow\left\Vert M\right\Vert _{0}$ of any quasi-colimit
diagram $\mathcal{D}$ of small objects\ over the set truncation $\left\Vert
M\right\Vert _{0}$\ of the type $M$\ is a limit diagram.of types.
\end{definition}

It is easy to see that

\begin{proposition}
\label{p3.1}(cf. Proposition 1 of \S 2.3 in \cite{la}) We have the following:

\begin{enumerate}
\item A type $M$\ is microlinear iff its set truncation $\left\Vert
M\right\Vert _{0}$ is so.

\item The type $\mathbb{R}$\ is microlinear.

\item If $M$ is a microlinear set and $X$ is an arbitrary type, then
$X\rightarrow M$ is a microlinear set.

\item If $M$ is the limit of a diagram $\mathcal{M}$\ of microlinear sets,
then $M$ is a microlinear set.
\end{enumerate}
\end{proposition}

\begin{proof}
The first statement follows directly from the very definition of
microlinearity. The second statement follows from the axiom. Let $\mathcal{D}%
$\ be a quasi-colimit diagram of small objects. For the third statement, we
note that the diagram%
\[
\mathcal{D}\rightarrow X\rightarrow M
\]
is equivalent to the diagram%
\[
X\rightarrow\mathcal{D}\rightarrow M
\]
which is a limit diagram because of the assumption that $\mathcal{D}%
\rightarrow M$ is a limit diagram. For the fourth statement, we note that the
diagram%
\[
\mathcal{D}\rightarrow\mathcal{M}%
\]
is a limit diagram of diagrams of types over the diagram $\mathcal{M}$\ so
that the diagram%
\[
\mathcal{D}\rightarrow M
\]
is a limit diagram, because, roughly speaking, double limits commute.
\end{proof}

\section{\label{s4}Tangency}

\begin{notation}
Given a microlinear type $M$\ and $x:M$, the type $\mathbf{T}_{x}M$\ of
\underline{tangent vectors to $M$ at $x$} stands for the subtype%
\[
\left\{  t:D\rightarrow\left\Vert M\right\Vert _{0}\mid t\left(  0\right)
=\left\vert x\right\vert _{0}\right\}
\]
of the type%
\[
D\rightarrow\left\Vert M\right\Vert _{0}%
\]

\end{notation}

We recall that.

\begin{lemma}
\label{l4.1}(cf. Proposition 6 of \S 2.2 in \cite{la}) The following diagram
is a quasi-colimit diagram:
\[%
\begin{array}
[c]{ccc}%
1 & \rightarrow & D\\
\downarrow &  & \downarrow\lambda_{d:D}\left(  0,d\right) \\
D & \overrightarrow{\lambda_{d:D}\left(  d,0\right)  } & D(2)
\end{array}
\]

\end{lemma}

\begin{corollary}
\label{cl4.1}Let $M$\ be a microlinear set with $x:M$. Given $t_{1}%
,t_{2}:D\rightarrow M$ with $t_{1}\left(  0\right)  =t_{2}\left(  0\right)
=x$, there exists $\mathfrak{l}_{\left(  t_{1},t_{2}\right)  }:D(2)\rightarrow
M$ such that
\begin{align*}
\mathfrak{l}_{\left(  t_{1},t_{2}\right)  }\circ\left(  \lambda_{d:D}\left(
d,0\right)  \right)   &  =t_{1}\\
\mathfrak{l}_{\left(  t_{1},t_{2}\right)  }\circ\left(  \lambda_{d:D}\left(
0,d\right)  \right)   &  =t_{2}%
\end{align*}

\end{corollary}

The above lemma has the following variant.

\begin{lemma}
\label{l4.2}The following diagram is a quasi-colimit diagram:%
\[%
\begin{array}
[c]{ccccc}
&  & 1 &  & \\
& \swarrow & \downarrow & \searrow & \\
D &  & D &  & D\\
& \searrow & \downarrow & \swarrow & \\
&  & D\left(  3\right)  &  &
\end{array}
\]
where the lower three arrows stand from left to right for%
\begin{align*}
&  \lambda_{d:D}\left(  d,0,0\right) \\
&  \lambda_{d:D}\left(  0,d,0\right) \\
&  \lambda_{d:D}\left(  0,0,d\right)
\end{align*}
respectively.
\end{lemma}

\begin{corollary}
\label{cl4.2}Let $M$\ be a microlinear set with $x:M$. Given $t_{1}%
,t_{2}:D\rightarrow M$ with $t_{1}\left(  0\right)  =t_{2}\left(  0\right)
=x$, there exists $\mathfrak{l}_{\left(  t_{1},t_{2},t_{3}\right)
}:D(3)\rightarrow M$ such that
\begin{align*}
\mathfrak{l}_{\left(  t_{1},t_{2},t_{3}\right)  }\circ\left(  \lambda
_{d:D}\left(  d,0,0\right)  \right)   &  =t_{1}\\
\mathfrak{l}_{\left(  t_{1},t_{2},t_{3}\right)  }\circ\left(  \lambda
_{d:D}\left(  0,d,0\right)  \right)   &  =t_{2}\\
\mathfrak{l}_{\left(  t_{1},t_{2},t_{3}\right)  }\circ\left(  \lambda
_{d:D}\left(  0,0,d\right)  \right)   &  =t_{3}%
\end{align*}

\end{corollary}

\begin{definition}
Given a microlinear type $M$\ with $x:M$, we define addition and scalar
multiplication on $\mathbf{T}_{x}M$\ as follows:For $t,t_{1},t_{2}%
:\mathbf{T}_{x}M$ and $\alpha:\mathbb{R}$, $t_{1}+t_{2}$\ and $\alpha t$\ are
defined to be%
\begin{align*}
t_{1}+t_{2}  &  :\equiv\lambda_{d:D}\mathfrak{l}_{\left(  t_{1},t_{2}\right)
}\left(  d,d\right) \\
\alpha t  &  :\equiv\lambda_{d:D}t\left(  \alpha d\right)
\end{align*}

\end{definition}

\begin{theorem}
\label{t4.1}Let $M$\ be a microlinear type with $x:M$. Given $\alpha
,\beta:\mathbb{R}$ and $t,t_{1},t_{2},t_{3}:\mathbf{T}_{x}(M)$, we have
\begin{align}
\left(  t_{1}+t_{2}\right)  +t_{3}  &  =t_{1}+\left(  t_{2}+t_{3}\right)
\label{t4.1.1}\\
t_{1}+t_{2}  &  =t_{2}+t_{1}\label{t4.1.2}\\
1t  &  =t\label{t4.1.5}\\
\left(  \alpha+\beta\right)  t  &  =\alpha t+\beta t\label{t4.1.6}\\
\alpha\left(  t_{1}+t_{2}\right)   &  =\alpha t_{1}+\alpha t_{2}%
\label{t4.1.7}\\
\left(  \alpha\beta\right)  t  &  =\alpha\left(  \beta t\right)
\label{t4.1.8}%
\end{align}
In a word, the type $\mathbf{T}_{x}(M)$ is an $\mathbb{R}$-module.
\end{theorem}

\begin{proof}
We deal with the six properties in order.

\begin{enumerate}
\item It is easy to see that%
\begin{align*}
\left(  \lambda_{\left(  d_{1},d_{2}\right)  :D\left(  2\right)  }%
\mathfrak{l}_{\left(  t_{1},t_{2},t_{3}\right)  }\left(  d_{1},d_{2},0\right)
\right)  \circ\left(  \lambda_{d:D}\left(  d,0\right)  \right)   &
=\lambda_{d:D}t_{1}\left(  d\right) \\
\left(  \lambda_{\left(  d_{1},d_{2}\right)  :D\left(  2\right)  }%
\mathfrak{l}_{\left(  t_{1},t_{2},t_{3}\right)  }\left(  d_{1},d_{2},0\right)
\right)  \circ\left(  \lambda_{d:D}\left(  0,d\right)  \right)   &
=\lambda_{d:D}t_{2}\left(  d\right)
\end{align*}
so that%
\[
\mathfrak{l}_{\left(  t_{1},t_{2}\right)  }=\lambda_{\left(  d_{1}%
,d_{2}\right)  :D\left(  2\right)  }\mathfrak{l}_{\left(  t_{1},t_{2}%
,t_{3}\right)  }\left(  d_{1},d_{2},0\right)
\]
and consequently%
\[
t_{1}+t_{2}=\lambda_{d:D}\mathfrak{l}_{\left(  t_{1},t_{2},t_{3}\right)
}\left(  d,d,0\right)
\]
It is easy to see that%
\begin{align*}
\left(  \lambda_{\left(  d_{1},d_{2}\right)  :D\left(  2\right)  }%
\mathfrak{l}_{\left(  t_{1},t_{2},t_{3}\right)  }\left(  d_{1},d_{1}%
,d_{2}\right)  \right)  \circ\left(  \lambda_{d:D}\left(  d,0\right)  \right)
&  =\lambda_{d:D}\mathfrak{l}_{\left(  t_{1},t_{2},t_{3}\right)  }\left(
d,d,0\right)  =t_{1}+t_{2}\\
\left(  \lambda_{\left(  d_{1},d_{2}\right)  :D\left(  2\right)  }%
\mathfrak{l}_{\left(  t_{1},t_{2},t_{3}\right)  }\left(  d_{1},d_{1}%
,d_{2}\right)  \right)  \circ\left(  \lambda_{d:D}\left(  0,d\right)  \right)
&  =\lambda_{d:D}\mathfrak{l}_{\left(  t_{1},t_{2},t_{3}\right)  }\left(
0,0,d\right)  =t_{3}%
\end{align*}
so that%
\[
\mathfrak{l}_{\left(  t_{1}+t_{2},t_{3}\right)  }=\lambda_{\left(  d_{1}%
,d_{2}\right)  :D\left(  2\right)  }\mathfrak{l}_{\left(  t_{1},t_{2}%
,t_{3}\right)  }\left(  d_{1},d_{1},d_{2}\right)
\]
and consequently%
\begin{equation}
\left(  t_{1}+t_{2}\right)  +t_{3}=\lambda_{d:D}\mathfrak{l}_{\left(
t_{1},t_{2},t_{3}\right)  }\left(  d,d,d\right)  \label{t4.1.9}%
\end{equation}
On the other hand, it is easy to see that%
\begin{align*}
\left(  \lambda_{\left(  d_{1},d_{2}\right)  :D\left(  2\right)  }%
\mathfrak{l}_{\left(  t_{1},t_{2},t_{3}\right)  }\left(  0,d_{1},d_{2}\right)
\right)  \circ\left(  \lambda_{d:D}\left(  d,0\right)  \right)   &
=\lambda_{d:D}t_{2}\left(  d\right) \\
\left(  \lambda_{\left(  d_{1},d_{2}\right)  :D\left(  2\right)  }%
\mathfrak{l}_{\left(  t_{1},t_{2},t_{3}\right)  }\left(  0,d_{1},d_{2}\right)
\right)  \circ\left(  \lambda_{d:D}\left(  0,d\right)  \right)   &
=\lambda_{d:D}t_{3}\left(  d\right)
\end{align*}
so that%
\[
\mathfrak{l}_{\left(  t_{2},t_{3}\right)  }=\lambda_{\left(  d_{1}%
,d_{2}\right)  :D\left(  2\right)  }\mathfrak{l}_{\left(  t_{1},t_{2}%
,t_{3}\right)  }\left(  0,d_{1},d_{2}\right)
\]
and consequently%
\[
t_{2}+t_{3}=\lambda_{d:D}\mathfrak{l}_{\left(  t_{1},t_{2},t_{3}\right)
}\left(  0,d,d\right)
\]
It is easy to see that%
\begin{align*}
\left(  \lambda_{\left(  d_{1},d_{2}\right)  :D\left(  2\right)  }%
\mathfrak{l}_{\left(  t_{1},t_{2},t_{3}\right)  }\left(  d_{1},d_{2}%
,d_{2}\right)  \right)  \circ\left(  \lambda_{d:D}\left(  d,0\right)  \right)
&  =\lambda_{d:D}\mathfrak{l}_{\left(  t_{1},t_{2},t_{3}\right)  }\left(
d,0,0\right)  =t_{1}\\
\left(  \lambda_{\left(  d_{1},d_{2}\right)  :D\left(  2\right)  }%
\mathfrak{l}_{\left(  t_{1},t_{2},t_{3}\right)  }\left(  d_{1},d_{2}%
,d_{2}\right)  \right)  \circ\left(  \lambda_{d:D}\left(  0,d\right)  \right)
&  =\lambda_{d:D}\mathfrak{l}_{\left(  t_{1},t_{2},t_{3}\right)  }\left(
0,d,d\right)  =t_{2}+t_{3}%
\end{align*}
so that%
\[
\mathfrak{l}_{\left(  t_{1},t_{2}+t_{3}\right)  }=\lambda_{\left(  d_{1}%
,d_{2}\right)  :D\left(  2\right)  }\mathfrak{l}_{\left(  t_{1},t_{2}%
,t_{3}\right)  }\left(  d_{1},d_{2},d_{2}\right)
\]
and consequently%
\begin{equation}
t_{1}+\left(  t_{2}+t_{3}\right)  =\lambda_{d:D}\mathfrak{l}_{\left(
t_{1},t_{2},t_{3}\right)  }\left(  d,d,d\right)  \label{t4.1.10}%
\end{equation}
It follows from (\ref{t4.1.9}) and (\ref{t4.1.10}) that (\ref{t4.1.1}) obtains.

\item It is easy to see that%
\begin{align*}
\left(  \lambda_{\left(  d_{1},d_{2}\right)  :D\left(  2\right)  }%
\mathfrak{l}_{\left(  t_{1},t_{2}\right)  }\left(  d_{2},d_{1}\right)
\right)  \circ\left(  \lambda_{d:D}\left(  d,0\right)  \right)   &
=\lambda_{d:D}t_{2}\left(  d\right) \\
\left(  \lambda_{\left(  d_{1},d_{2}\right)  :D\left(  2\right)  }%
\mathfrak{l}_{\left(  t_{1},t_{2}\right)  }\left(  d_{2},d_{1}\right)
\right)  \circ\left(  \lambda_{d:D}\left(  0,d\right)  \right)   &
=\lambda_{d:D}t_{1}\left(  d\right)
\end{align*}
so that%
\[
\mathfrak{l}_{\left(  t_{2},t_{1}\right)  }=\lambda_{\left(  d_{1}%
,d_{2}\right)  :D\left(  2\right)  }\mathfrak{l}_{\left(  t_{1},t_{2}\right)
}\left(  d_{2},d_{1}\right)
\]
Therefore we have%
\begin{align*}
t_{2}+t_{1}  &  =\lambda_{d:D}\mathfrak{l}_{\left(  t_{2},t_{1}\right)
}\left(  d,d\right) \\
&  =\lambda_{d:D}\mathfrak{l}_{\left(  t_{1},t_{2}\right)  }\left(  d,d\right)
\\
&  =\lambda_{d:D}\left(  t_{1}+t_{2}\right)  \left(  d\right)
\end{align*}
so that (\ref{t4.1.2}) obtains.

\item It is easy to see that, for any $d:D$, we have%
\[
\left(  1t\right)  \left(  d\right)  =t\left(  1d\right)  =t\left(  d\right)
\]
so that (\ref{t4.1.5}) obtains.

\item It is easy to see that%
\begin{align*}
\left(  \lambda_{\left(  d_{1},d_{2}\right)  :D\left(  2\right)  }t\left(
\alpha d_{1}+\beta d_{2}\right)  \right)  \circ\left(  \lambda_{d:D}\left(
d,0\right)  \right)   &  =\lambda_{d:D}t\left(  \alpha d\right)
=\lambda_{d:D}\left(  \alpha t\right)  \left(  d\right) \\
\left(  \lambda_{\left(  d_{1},d_{2}\right)  :D\left(  2\right)  }t\left(
\alpha d_{1}+\beta d_{2}\right)  \right)  \circ\left(  \lambda_{d:D}\left(
0,d\right)  \right)   &  =\lambda_{d:D}t\left(  \beta d\right)  =\lambda
_{d:D}\left(  \beta t\right)  \left(  d\right)
\end{align*}
so that%
\[
\mathfrak{l}_{\left(  \alpha t,\beta t\right)  }=\lambda_{\left(  d_{1}%
,d_{2}\right)  :D\left(  2\right)  }t\left(  \alpha d_{1}+\beta d_{2}\right)
\]
Therefore, for any $d:D$, we have%
\begin{align*}
\left(  \alpha+\beta\right)  t  &  =\lambda_{d:D}t\left(  \left(  \alpha
+\beta\right)  d\right) \\
&  =\lambda_{d:D}t\left(  \alpha d+\beta d\right) \\
&  =\lambda_{d:D}\mathfrak{l}_{\left(  \alpha t,\beta t\right)  }\left(
d,d\right) \\
&  =\lambda_{d:D}\left(  \alpha t+\beta t\right)  \left(  d\right)
\end{align*}
so that (\ref{t4.1.6}) obtains.

\item It is easy to see that%
\begin{align*}
\left(  \lambda_{\left(  d_{1},d_{2}\right)  :D\left(  2\right)  }%
\mathfrak{l}_{\left(  t_{1},t_{2}\right)  }\left(  \alpha d_{1},\alpha
d_{2}\right)  \right)  \circ\left(  \lambda_{d:D}\left(  d,0\right)  \right)
&  =\lambda_{d:D}t_{1}\left(  \alpha d\right)  =\lambda_{d:D}\left(  \alpha
t_{1}\right)  \left(  d\right) \\
\left(  \lambda_{\left(  d_{1},d_{2}\right)  :D\left(  2\right)  }%
\mathfrak{l}_{\left(  t_{1},t_{2}\right)  }\left(  \alpha d_{1},\alpha
d_{2}\right)  \right)  \circ\left(  \lambda_{d:D}\left(  0,d\right)  \right)
&  =\lambda_{d:D}t_{2}\left(  \alpha d\right)  =\lambda_{d:D}\left(  \alpha
t_{2}\right)  \left(  d\right)
\end{align*}
so that%
\[
\mathfrak{l}_{\left(  \alpha t_{1},\alpha t_{2}\right)  }=\lambda_{\left(
d_{1},d_{2}\right)  :D\left(  2\right)  }\mathfrak{l}_{\left(  t_{1}%
,t_{2}\right)  }\left(  \alpha d_{1},\alpha d_{2}\right)
\]
Therefore, for any $d:D$, we have%
\begin{align*}
\alpha\left(  t_{1}+t_{2}\right)   &  =\alpha\left(  \lambda_{d:D}%
\mathfrak{l}_{\left(  t_{1},t_{2}\right)  }\left(  d.d\right)  \right) \\
&  =\lambda_{d:D}\mathfrak{l}_{\left(  t_{1},t_{2}\right)  }\left(  \alpha
d.\alpha d\right) \\
&  =\lambda_{d:D}\mathfrak{l}_{\left(  \alpha t_{1},\alpha t_{2}\right)
}\left(  d.d\right) \\
&  =\lambda_{d:D}\left(  \alpha t_{1}+\alpha t_{2}\right)  \left(  d\right)
\end{align*}
so that (\ref{t4.1.7}) obtains.

\item It is easy to see that%
\begin{align*}
\left(  \alpha\beta\right)  t  &  =\lambda_{d:D}t\left(  \alpha\beta d\right)
\\
&  =\lambda_{d:D}\left(  \beta t\right)  \left(  \alpha d\right) \\
&  =\lambda_{d:D}\left(  \alpha\left(  \beta t\right)  \right)  \left(
d\right)
\end{align*}
so that (\ref{t4.1.8}) obtains.
\end{enumerate}
\end{proof}

We recall that

\begin{lemma}
\label{l6.1}(cf.Proposition 7 of \S 2.2 in \cite{la}) The following is a
quasi-colimit diagram:%
\[%
\begin{array}
[c]{ccccc}
& \underrightarrow{\lambda_{d:D}\left(  d,0\right)  } &  &  & \\
D & \underrightarrow{\lambda_{d:D}\left(  0,d\right)  } & D\times D &
\underrightarrow{\lambda_{\left(  d_{1},d_{2}\right)  :D\times D}d_{1}d_{2}} &
D\\
& \underrightarrow{\lambda_{d:D}\left(  0,0\right)  } &  &  &
\end{array}
\]

\end{lemma}

\begin{corollary}
\label{cl6.1}Let $M$ be a microlinear set. Given%
\[
\theta:D\times D\rightarrow M
\]
in accordance with%
\[
\theta\circ\left(  \lambda_{d:D}\left(  d,0\right)  \right)  =\theta
\circ\left(  \lambda_{d:D}\left(  0,d\right)  \right)  =\theta\circ\left(
\lambda_{d:D}\left(  0,0\right)  \right)
\]
there exists a homotopically unique%
\[
t:D\rightarrow M
\]
in accordance with%
\[
t\circ\left(  \lambda_{\left(  d_{1},d_{2}\right)  :D\times D}d_{1}%
d_{2}\right)  =\theta
\]

\end{corollary}

\begin{theorem}
\label{t4.2}(cf. Proposition 2 of \S 3.1.1 in \cite{la}) Let $M$\ be a
microlinear type. For any $x:M$, the $\mathbb{R}$-module $\mathbf{T}_{x}%
(M)$\ is Euclidean.
\end{theorem}

\begin{proof}
It is easy to see that%
\begin{align*}
\left(  \lambda_{\left(  d_{1},d_{2}\right)  :D\times D}\left(  \varphi\left(
d_{1}\right)  -\varphi\left(  0\right)  \right)  \left(  d_{2}\right)
\right)  \circ\left(  \lambda_{d:D}\left(  d,0\right)  \right)   &
=\lambda_{d:D}\left\vert x\right\vert _{0}\\
\left(  \lambda_{\left(  d_{1},d_{2}\right)  :D\times D}\left(  \varphi\left(
d_{1}\right)  -\varphi\left(  0\right)  \right)  \left(  d_{2}\right)
\right)  \circ\left(  \lambda_{d:D}\left(  0,d\right)  \right)   &
=\lambda_{d:D}\left\vert x\right\vert _{0}\\
\left(  \lambda_{\left(  d_{1},d_{2}\right)  :D\times D}\left(  \varphi\left(
d_{1}\right)  -\varphi\left(  0\right)  \right)  \left(  d_{2}\right)
\right)  \circ\left(  \lambda_{d:D}\left(  0,0\right)  \right)   &
=\lambda_{d:D}\left\vert x\right\vert _{0}%
\end{align*}
Therefore, by dint of Corollary \ref{cl6.1}, there exists $t:D\rightarrow
\left\Vert M\right\Vert _{0}$ such that%
\[
\lambda_{\left(  d_{1},d_{2}\right)  :D\times D}\left(  \varphi\left(
d_{1}\right)  -\varphi\left(  0\right)  \right)  \left(  d_{2}\right)
=\lambda_{\left(  d_{1},d_{2}\right)  :D\times D}t\left(  d_{1}d_{2}\right)
\]
which is no other than%
\[%
{\displaystyle\prod\limits_{d:D}}
\varphi\left(  d\right)  -\varphi\left(  0\right)  =dt
\]
This completes the proof.
\end{proof}

\section{\label{s5}Strong Differences}

We recall that

\begin{lemma}
\label{l5.l}(The first Lemma of \S 3.4 in \cite{la}) The following diagram is
a quasi-colimit diagram:
\[%
\begin{array}
[c]{ccc}%
D^{2}\left\{  \left(  1,2\right)  \right\}  & \underrightarrow{\lambda
_{\left(  d_{1},d_{2}\right)  :D^{2}\left\{  \left(  1,2\right)  \right\}
}\left(  d_{1},d_{2}\right)  } & D^{2}\\
\lambda_{\left(  d_{1},d_{2}\right)  :D^{2}\left\{  \left(  1,2\right)
\right\}  }\left(  d_{1},d_{2}\right)  \downarrow &  & \downarrow
\lambda_{\left(  d_{1},d_{2}\right)  :D^{2}}\left(  d_{1},d_{2},d_{1}%
d_{2}\right) \\
D^{2} & \overrightarrow{\lambda_{\left(  d_{1},d_{2}\right)  :D^{2}}\left(
d_{1},d_{2},0\right)  } & D^{3}\left\{  \left(  1,3\right)  ,\left(
2,3\right)  \right\}
\end{array}
\]

\end{lemma}

\begin{corollary}
\label{cl5.1}Let $M$\ be a microlinear set. Given $\theta_{1},\theta_{2}%
:D^{2}\rightarrow M$ with
\[
\theta_{1}\circ\left(  \lambda_{\left(  d_{1},d_{2}\right)  :D^{2}\left\{
\left(  1,2\right)  \right\}  }\left(  d_{1},d_{2}\right)  \right)
=\theta_{2}\circ\left(  \lambda_{\left(  d_{1},d_{2}\right)  :D^{2}\left\{
\left(  1,2\right)  \right\}  }\left(  d_{1},d_{2}\right)  \right)
\]
there exists $\mathfrak{m}_{\left(  \theta_{1},\theta_{2}\right)  }%
:D^{3}\left\{  \left(  1,3\right)  ,\left(  2,3\right)  \right\}  \rightarrow
M$ with
\begin{align*}
\mathfrak{m}_{\left(  \theta_{1},\theta_{2}\right)  }\circ\left(
\lambda_{\left(  d_{1},d_{2}\right)  :D^{2}}\left(  d_{1},d_{2},0\right)
\right)   &  =\theta_{2}\\
\mathfrak{m}_{\left(  \theta_{1},\theta_{2}\right)  }\circ\left(
\lambda_{\left(  d_{1},d_{2}\right)  :D^{2}}\left(  d_{1},d_{2},d_{1}%
d_{2}\right)  \right)   &  =\theta_{1}%
\end{align*}

\end{corollary}

Now we define strong differences.

\begin{definition}
Let $M$\ be a microlinear set. Given $\theta_{1},\theta_{2}:D^{2}\rightarrow
M$ with
\[
\theta_{1}\circ\left(  \lambda_{\left(  d_{1},d_{2}\right)  :D^{2}\left\{
\left(  1,2\right)  \right\}  }\left(  d_{1},d_{2}\right)  \right)
=\theta_{2}\circ\left(  \lambda_{\left(  d_{1},d_{2}\right)  :D^{2}\left\{
\left(  1,2\right)  \right\}  }\left(  d_{1},d_{2}\right)  \right)
\]
we define $\theta_{1}\overset{\cdot}{-}\theta_{2}:D\rightarrow M$ to be
\[
\theta_{1}\overset{\cdot}{-}\theta_{2}:\equiv\lambda_{d:D}\mathfrak{m}%
_{\left(  \theta_{1},\theta_{2}\right)  }\left(  0,0,d\right)
\]

\end{definition}

We recall that

\begin{proposition}
\label{p5.1}(cf. Proposition 8 of \S 3.4 in \cite{la}) Let $M$\ be a
microlinear set. Given $\theta_{1},\theta_{2}:D^{2}\rightarrow M$ with
\[
\theta_{1}\circ\left(  \lambda_{\left(  d_{1},d_{2}\right)  :D^{2}\left\{
\left(  1,2\right)  \right\}  }\left(  d_{1},d_{2}\right)  \right)
=\theta_{2}\circ\left(  \lambda_{\left(  d_{1},d_{2}\right)  :D^{2}\left\{
\left(  1,2\right)  \right\}  }\left(  d_{1},d_{2}\right)  \right)
\]
we have
\begin{align*}
&  \theta_{1}\circ\left(  \lambda_{\left(  d_{1},d_{2}\right)  :D^{2}}\left(
d_{2},d_{1}\right)  \right)  \circ\left(  \lambda_{\left(  d_{1},d_{2}\right)
:D^{2}\left\{  \left(  1,2\right)  \right\}  }\left(  d_{1},d_{2}\right)
\right) \\
&  =\theta_{2}\circ\left(  \lambda_{\left(  d_{1},d_{2}\right)  :D^{2}}\left(
d_{2},d_{1}\right)  \right)  \circ\left(  \lambda_{\left(  d_{1},d_{2}\right)
:D^{2}\left\{  \left(  1,2\right)  \right\}  }\left(  d_{1},d_{2}\right)
\right)
\end{align*}
and
\begin{align*}
&  \theta_{1}\circ\left(  \lambda_{\left(  d_{1},d_{2}\right)  :D^{2}}\left(
d_{2},d_{1}\right)  \right)  \overset{\cdot}{-}\theta_{2}\circ\left(
\lambda_{\left(  d_{1},d_{2}\right)  :D^{2}}\left(  d_{2},d_{1}\right)
\right) \\
&  =\theta_{1}\overset{\cdot}{-}\theta_{2}%
\end{align*}

\end{proposition}

\begin{proof}
The first identity should be obvious. For the second identity, it suffices to
note that%
\[
\mathfrak{m}_{\left(  \theta_{1}\circ\left(  \lambda_{\left(  d_{1}%
,d_{2}\right)  :D^{2}}\left(  d_{2},d_{1}\right)  \right)  ,\theta_{2}%
\circ\left(  \lambda_{\left(  d_{1},d_{2}\right)  :D^{2}}\left(  d_{2}%
,d_{1}\right)  \right)  \right)  }=\lambda_{\left(  d_{1},d_{2},d_{3}\right)
:D^{3}\left\{  \left(  1,3\right)  ,\left(  2,3\right)  \right\}
}\mathfrak{m}_{\left(  \theta_{1},\theta_{2}\right)  }\left(  d_{2}%
,d_{1},d_{3}\right)
\]

\end{proof}

\begin{definition}
Let $M$\ be a microlinear set. We give two definitions:

\begin{itemize}
\item Given $\theta_{1},\theta_{2}:D^{2}\rightarrow M$ with
\[
\theta_{1}\circ\left(  \lambda_{d:D}\left(  0,d\right)  \right)  =\theta
_{2}\circ\left(  \lambda_{d:D}\left(  0,d\right)  \right)
\]
we define $\theta_{1}\underset{1}{+}\theta_{2}:D^{2}\rightarrow M$\ to be
\[
\theta_{1}\underset{1}{+}\theta_{2}:\equiv\lambda_{\left(  d_{1},d_{2}\right)
:D^{2}}\left(  \lambda_{d_{1}:D}\lambda_{d_{2}:D}\theta_{1}\left(  d_{1}%
,d_{2}\right)  +\lambda_{d_{1}:D}\lambda_{d_{2}:D}\theta_{2}\left(
d_{1},d_{2}\right)  \right)  \left(  d_{1}\right)  \left(  d_{2}\right)
\]

\item Given $\alpha:\mathbb{R}$ and $\theta:D^{2}\rightarrow M$, we define
$\alpha\underset{1}{\cdot}\theta$\ to be
\[
\alpha\underset{1}{\cdot}\theta:\equiv\lambda_{\left(  d_{1},d_{2}\right)
:D^{2}}\left(  \alpha\left(  \lambda_{d_{1}:D}\lambda_{d_{2}:D}\theta\left(
d_{1},d_{2}\right)  \right)  \right)  \left(  d_{1}\right)  \left(
d_{2}\right)
\]

\end{itemize}
\end{definition}

\begin{lemma}
\label{l5.2}The diagram consisting of
\begin{align*}
\mathcal{N}  &  :\equiv D^{5}\left\{  \left(  1,2\right)  ,\left(  1,4\right)
,\left(  1,5\right)  ,\left(  2,4\right)  ,\left(  2,5\right)  ,\left(
4,5\right)  \right\} \\
\mathcal{L}_{11}  &  :\equiv D^{2},\mathcal{L}_{12}:\equiv D^{2}%
,\mathcal{L}_{21}:\equiv D^{2},\mathcal{L}_{22}:\equiv D^{2}\\
\mathcal{P}_{1}  &  :\equiv D^{2}\left\{  \left(  1,2\right)  \right\}
,\mathcal{P}_{2}:\equiv D^{2}\left\{  \left(  1,2\right)  \right\} \\
\mathcal{Q}_{1}  &  :\equiv D,\mathcal{Q}_{2}:\equiv D\\
\lambda_{\left(  d_{1},d_{2}\right)  :D^{2}}\left(  d_{1},0,d_{2},d_{1}%
d_{2},0\right)   &  :\mathcal{L}_{11}\rightarrow\mathcal{N},\lambda_{\left(
d_{1},d_{2}\right)  :D^{2}}\left(  d_{1},0,d_{2},0,0\right)  :\mathcal{L}%
_{21}\rightarrow\mathcal{N}\\
\lambda_{\left(  d_{1},d_{2}\right)  :D^{2}}\left(  0,d_{1},d_{2},0,d_{1}%
d_{2}\right)   &  :\mathcal{L}_{12}\rightarrow\mathcal{N},\lambda_{\left(
d_{1},d_{2}\right)  :D^{2}}\left(  0,d_{1},d_{2},0,0\right)  :\mathcal{L}%
_{22}\rightarrow\mathcal{N}\\
\lambda_{\left(  d_{1},d_{2}\right)  :D^{2}\left\{  \left(  1,2\right)
\right\}  }\left(  d_{1},d_{2}\right)   &  :\mathcal{P}_{1}\rightarrow
\mathcal{L}_{11},\lambda_{\left(  d_{1},d_{2}\right)  :D^{2}\left\{  \left(
1,2\right)  \right\}  }\left(  d_{1},d_{2}\right)  :\mathcal{P}_{1}%
\rightarrow\mathcal{L}_{21}\\
\lambda_{\left(  d_{1},d_{2}\right)  :D^{2}\left\{  \left(  1,2\right)
\right\}  }\left(  d_{1},d_{2}\right)   &  :\mathcal{P}_{2}\rightarrow
\mathcal{L}_{12},\lambda_{\left(  d_{1},d_{2}\right)  :D^{2}\left\{  \left(
1,2\right)  \right\}  }\left(  d_{1},d_{2}\right)  :\mathcal{P}_{2}%
\rightarrow\mathcal{L}_{22}\\
\lambda_{d:D}\left(  0,d\right)   &  :\mathcal{Q}_{1}\rightarrow
\mathcal{L}_{11},\lambda_{d:D}\left(  0,d\right)  :\mathcal{Q}_{1}%
\rightarrow\mathcal{L}_{12}\\
\lambda_{d:D}\left(  0,d\right)   &  :\mathcal{Q}_{2}\rightarrow
\mathcal{L}_{21},\lambda_{d:D}\left(  0,d\right)  :\mathcal{Q}_{2}%
\rightarrow\mathcal{L}_{22}%
\end{align*}
is a quasi-colimit diagram.
\end{lemma}

\begin{proof}
By Axiom \ref{a1.1} we are sure that, given
\[
\gamma^{11},\gamma^{12},\gamma^{21},\gamma^{22}:D^{2}\rightarrow\mathbb{R}%
\]
there exist
\[
a^{11},a_{1}^{11},a_{2}^{11},a_{12}^{11},a^{12},a_{1}^{12},a_{2}^{12}%
,a_{12}^{12},a^{21},a_{1}^{21},a_{2}^{21},a_{12}^{21},a^{22},a_{1}^{22}%
,a_{2}^{22},a_{12}^{22}:\mathbb{R}%
\]
such that
\begin{align*}
&  \lambda_{\left(  d_{1},d_{2}\right)  :D^{2}}\gamma^{11}\left(  d_{1}%
,d_{2}\right) \\
&  =\lambda_{\left(  d_{1},d_{2}\right)  :D^{2}}\left(  a^{11}+a_{1}^{11}%
d_{1}+a_{2}^{11}d_{2}+a_{12}^{11}d_{1}d_{2}\right) \\
&  \lambda_{\left(  d_{1},d_{2}\right)  :D^{2}}\gamma^{12}\left(  d_{1}%
,d_{2}\right) \\
&  =\lambda_{\left(  d_{1},d_{2}\right)  :D^{2}}\left(  a^{12}+a_{1}^{12}%
d_{1}+a_{2}^{12}d_{2}+a_{12}^{12}d_{1}d_{2}\right) \\
&  \lambda_{\left(  d_{1},d_{2}\right)  :D^{2}}\gamma^{21}\left(  d_{1}%
,d_{2}\right) \\
&  =\lambda_{\left(  d_{1},d_{2}\right)  :D^{2}}\left(  a^{21}+a_{1}^{21}%
d_{1}+a_{2}^{21}d_{2}+a_{12}^{21}d_{1}d_{2}\right) \\
&  \lambda_{\left(  d_{1},d_{2}\right)  :D^{2}}\gamma^{22}\left(  d_{1}%
,d_{2}\right) \\
&  =\lambda_{\left(  d_{1},d_{2}\right)  :D^{2}}\left(  a^{22}+a_{1}^{22}%
d_{1}+a_{2}^{22}d_{2}+a_{12}^{22}d_{1}d_{2}\right)
\end{align*}
The conditions
\begin{align*}
\gamma^{11}\circ\left(  \lambda_{\left(  d_{1},d_{2}\right)  ::D^{2}\left\{
\left(  1,2\right)  \right\}  }\left(  d_{1},d_{2}\right)  \right)   &
=\gamma^{21}\circ\left(  \lambda_{\left(  d_{1},d_{2}\right)  ::D^{2}\left\{
\left(  1,2\right)  \right\}  }\left(  d_{1},d_{2}\right)  \right) \\
\gamma^{12}\circ\left(  \lambda_{\left(  d_{1},d_{2}\right)  ::D^{2}\left\{
\left(  1,2\right)  \right\}  }\left(  d_{1},d_{2}\right)  \right)   &
=\gamma^{22}\circ\left(  \lambda_{\left(  d_{1},d_{2}\right)  ::D^{2}\left\{
\left(  1,2\right)  \right\}  }\left(  d_{1},d_{2}\right)  \right)
\end{align*}
imply that
\begin{align*}
a^{11}  &  =a^{21},a_{1}^{11}=a_{1}^{21},a_{2}^{11}=a_{2}^{21}\\
a^{12}  &  =a^{22},a_{1}^{12}=a_{1}^{22},a_{2}^{12}=a_{2}^{22}%
\end{align*}
The conditions
\begin{align*}
\gamma^{11}\circ\left(  \lambda_{d:D}\left(  0,d\right)  \right)   &
=\gamma^{12}\circ\left(  \lambda_{d:D}\left(  0,d\right)  \right) \\
\gamma^{21}\circ\left(  \lambda_{d:D}\left(  0,d\right)  \right)   &
=\gamma^{22}\circ\left(  \lambda_{d:D}\left(  0,d\right)  \right)
\end{align*}
imply that
\begin{align*}
a^{11}  &  =a^{12},a_{2}^{11}=a_{2}^{12}\\
a^{21}  &  =a^{22},a_{2}^{21}=a_{2}^{22}%
\end{align*}
Therefore we have
\begin{align*}
a^{11}  &  =a^{12}=a^{21}=a^{22}\\
a_{1}^{11}  &  =a_{1}^{21},a_{1}^{12}=a_{1}^{22}\\
a_{2}^{11}  &  =a_{2}^{12}=a_{2}^{21}=a_{2}^{22}%
\end{align*}
which implies that there exist
\[
b,b_{1},b_{2},b_{3},b_{13},b_{23},b_{4},b_{5}:\mathbb{R}%
\]
such that
\begin{align*}
&  \left(  \lambda_{\left(  d_{1},d_{2},d_{3},d_{4},d_{5}\right)  :N}%
b+b_{1}d_{1}+b_{2}d_{2}+b_{3}d_{3}+b_{13}d_{1}d_{3}+b_{23}d_{2}d_{3}%
+b_{4}d_{4}+b_{5}d_{5}\right) \\
&  \circ\left(  \lambda_{\left(  d_{1},d_{2}\right)  :D^{2}}\left(
d_{1},0,d_{2},d_{1}d_{2},0\right)  \right) \\
&  =\gamma^{11}\\
&  \left(  \lambda_{\left(  d_{1},d_{2},d_{3},d_{4},d_{5}\right)  :N}%
b+b_{1}d_{1}+b_{2}d_{2}+b_{3}d_{3}+b_{13}d_{1}d_{3}+b_{23}d_{2}d_{3}%
+b_{4}d_{4}+b_{5}d_{5}\right) \\
&  \circ\left(  \lambda_{\left(  d_{1},d_{2}\right)  :D^{2}}\left(
d_{1},0,d_{2},0,0\right)  \right) \\
&  =\gamma^{21}\\
&  \left(  \lambda_{\left(  d_{1},d_{2},d_{3},d_{4},d_{5}\right)  :N}%
b+b_{1}d_{1}+b_{2}d_{2}+b_{3}d_{3}+b_{13}d_{1}d_{3}+b_{23}d_{2}d_{3}%
+b_{4}d_{4}+b_{5}d_{5}\right) \\
&  \circ\left(  \lambda_{\left(  d_{1},d_{2}\right)  :D^{2}}\left(
0,d_{1},d_{2},0,d_{1}d_{2}\right)  \right) \\
&  =\gamma^{12}\\
&  \left(  \lambda_{\left(  d_{1},d_{2},d_{3},d_{4},d_{5}\right)  :N}%
b+b_{1}d_{1}+b_{2}d_{2}+b_{3}d_{3}+b_{13}d_{1}d_{3}+b_{23}d_{2}d_{3}%
+b_{4}d_{4}+b_{5}d_{5}\right) \\
&  \circ\left(  \lambda_{\left(  d_{1},d_{2}\right)  :D^{2}}\left(
0,d_{1},d_{2},0,0\right)  \right) \\
&  =\gamma^{22}%
\end{align*}
This completes the proof.
\end{proof}

\begin{corollary}
\label{cl5.2}Let $M$\ be a microlinear set. Given $\theta_{11},\theta
_{12},\theta_{21},\theta_{22}:D^{2}\rightarrow M$ with
\begin{align*}
\theta_{11}\circ\left(  \lambda_{\left(  d_{1},d_{2}\right)  :D^{2}\left\{
\left(  1,2\right)  \right\}  }\left(  d_{1},d_{2}\right)  \right)   &
=\theta_{21}\circ\left(  \lambda_{\left(  d_{1},d_{2}\right)  :D^{2}\left\{
\left(  1,2\right)  \right\}  }\left(  d_{1},d_{2}\right)  \right) \\
\theta_{12}\circ\left(  \lambda_{\left(  d_{1},d_{2}\right)  :D^{2}\left\{
\left(  1,2\right)  \right\}  }\left(  d_{1},d_{2}\right)  \right)   &
=\theta_{22}\circ\left(  \lambda_{\left(  d_{1},d_{2}\right)  :D^{2}\left\{
\left(  1,2\right)  \right\}  }\left(  d_{1},d_{2}\right)  \right) \\
\theta_{11}\circ\left(  \lambda_{d:D}\left(  0,d\right)  \right)   &
=\theta_{12}\circ\left(  \lambda_{d:D}\left(  0,d\right)  \right) \\
\theta_{21}\circ\left(  \lambda_{d:D}\left(  0,d\right)  \right)   &
=\theta_{22}\circ\left(  \lambda_{d:D}\left(  0,d\right)  \right)
\end{align*}
there exists
\[
\mathfrak{n}_{\left(  \theta_{11},\theta_{12},\theta_{21},\theta_{22}\right)
}:\mathcal{N}\rightarrow M
\]
such that
\begin{align*}
\mathfrak{n}_{\left(  \theta_{11},\theta_{12},\theta_{21},\theta_{22}\right)
}\circ\left(  \lambda_{\left(  d_{1},d_{2}\right)  :D^{2}}\left(
d_{1},0,d_{2},d_{1}d_{2},0\right)  \right)   &  =\theta_{11}\\
\mathfrak{n}_{\left(  \theta_{11},\theta_{12},\theta_{21},\theta_{22}\right)
}\circ\left(  \lambda_{\left(  d_{1},d_{2}\right)  :D^{2}}\left(
d_{1},0,d_{2},0,0\right)  \right)   &  =\theta_{21}\\
\mathfrak{n}_{\left(  \theta_{11},\theta_{12},\theta_{21},\theta_{22}\right)
}\circ\left(  \lambda_{\left(  d_{1},d_{2}\right)  :D^{2}}\left(
0,d_{1},d_{2},0,d_{1}d_{2}\right)  \right)   &  =\theta_{12}\\
\mathfrak{n}_{\left(  \theta_{11},\theta_{12},\theta_{21},\theta_{22}\right)
}\circ\left(  \lambda_{\left(  d_{1},d_{2}\right)  :D^{2}}\left(
0,d_{1},d_{2},0,0\right)  \right)   &  =\theta_{22}%
\end{align*}

\end{corollary}

\begin{proposition}
\label{p5.2}Let $M$\ be a microlinear set. Given $\theta_{11},\theta
_{12},\theta_{21},\theta_{22}:D^{2}\rightarrow M$ with
\begin{align*}
\theta_{11}\circ\left(  \lambda_{\left(  d_{1},d_{2}\right)  :D^{2}\left\{
\left(  1,2\right)  \right\}  }\left(  d_{1},d_{2}\right)  \right)   &
=\theta_{21}\circ\left(  \lambda_{\left(  d_{1},d_{2}\right)  :D^{2}\left\{
\left(  1,2\right)  \right\}  }\left(  d_{1},d_{2}\right)  \right) \\
\theta_{12}\circ\left(  \lambda_{\left(  d_{1},d_{2}\right)  :D^{2}\left\{
\left(  1,2\right)  \right\}  }\left(  d_{1},d_{2}\right)  \right)   &
=\theta_{22}\circ\left(  \lambda_{\left(  d_{1},d_{2}\right)  :D^{2}\left\{
\left(  1,2\right)  \right\}  }\left(  d_{1},d_{2}\right)  \right) \\
\theta_{11}\circ\left(  \lambda_{d:D}\left(  0,d\right)  \right)   &
=\theta_{12}\circ\left(  \lambda_{d:D}\left(  0,d\right)  \right) \\
\theta_{21}\circ\left(  \lambda_{d:D}\left(  0,d\right)  \right)   &
=\theta_{22}\circ\left(  \lambda_{d:D}\left(  0,d\right)  \right)
\end{align*}
we have
\[
\left(  \theta_{11}\underset{1}{+}\theta_{12}\right)  \circ\left(
\lambda_{\left(  d_{1},d_{2}\right)  :D^{2}\left\{  \left(  1,2\right)
\right\}  }\left(  d_{1},d_{2}\right)  \right)  =\left(  \theta_{21}%
\underset{1}{+}\theta_{22}\right)  \circ\left(  \lambda_{\left(  d_{1}%
,d_{2}\right)  :D^{2}\left\{  \left(  1,2\right)  \right\}  }\left(
d_{1},d_{2}\right)  \right)
\]
and
\[
\left(  \theta_{11}\underset{1}{+}\theta_{12}\right)  \overset{\cdot}%
{-}\left(  \theta_{21}\underset{1}{+}\theta_{22}\right)  =\left(  \theta
_{11}\overset{\cdot}{-}\theta_{21}\right)  +\left(  \theta_{12}\overset{\cdot
}{-}\theta_{22}\right)
\]

\end{proposition}

\begin{proof}
Since
\begin{align*}
&  \lambda_{\left(  d_{1},d_{2},d_{3}\right)  :D^{3}\left\{  \left(
1,3\right)  ,\left(  2,3\right)  \right\}  }\mathfrak{m}_{\left(  \theta
_{11}\underset{1}{+}\theta_{12},\theta_{21}\underset{1}{+}\theta_{22}\right)
}\left(  d_{1},d_{2},d_{3}\right) \\
&  =\lambda_{\left(  d_{1},d_{2},d_{3}\right)  :D^{3}\left\{  \left(
1,3\right)  ,\left(  2,3\right)  \right\}  }\mathfrak{n}_{\left(  \theta
_{11},\theta_{12},\theta_{21},\theta_{22}\right)  }\left(  d_{1},d_{1}%
,d_{2},d_{3},d_{3}\right)
\end{align*}
and
\begin{align*}
&  \lambda_{\left(  d_{1},d_{2}\right)  :D^{2}\left\{  \left(  1,2\right)
\right\}  }\mathfrak{l}_{\left(  \theta_{11}\overset{\cdot}{-}\theta
_{21},\theta_{12}\overset{\cdot}{-}\theta_{22}\right)  }\left(  d_{1}%
,d_{2}\right) \\
&  =\lambda_{\left(  d_{1},d_{2}\right)  :D^{2}\left\{  \left(  1,2\right)
\right\}  }\mathfrak{n}_{\left(  \theta_{11},\theta_{12},\theta_{21}%
,\theta_{22}\right)  }\left(  0,0,0,d_{1},d_{2}\right)
\end{align*}
we have
\begin{align*}
&  \lambda_{d:D}\left(  \left(  \theta_{11}\underset{1}{+}\theta_{12}\right)
\overset{\cdot}{-}\left(  \theta_{21}\underset{1}{+}\theta_{22}\right)
\right)  \left(  d\right) \\
&  =\lambda_{d:D}\mathfrak{m}_{\left(  \theta_{11}\underset{1}{+}\theta
_{12},\theta_{21}\underset{1}{+}\theta_{22}\right)  }\left(  0,0,d\right) \\
&  =\lambda_{d:D}\mathfrak{n}_{\left(  \theta_{11},\theta_{12},\theta
_{21},\theta_{22}\right)  }\left(  0,0,0,d,d\right) \\
&  =\lambda_{d:D}\mathfrak{l}_{\left(  \theta_{11}\overset{\cdot}{-}%
\theta_{21},\theta_{12}\overset{\cdot}{-}\theta_{22}\right)  }\left(
d,d\right) \\
&  =\lambda_{d:D}\left(  \left(  \theta_{11}\overset{\cdot}{-}\theta
_{21}\right)  +\left(  \theta_{12}\overset{\cdot}{-}\theta_{22}\right)
\right)  \left(  d\right)
\end{align*}
This completes the proof.
\end{proof}

\begin{proposition}
\label{p5.3}Let $M$\ be a microlinear set. Given $\alpha:\mathbb{R}%
$\ and$\ \theta_{1},\theta_{2}:D^{2}\rightarrow M$ with
\[
\theta_{1}\circ\left(  \lambda_{\left(  d_{1},d_{2}\right)  :D^{2}\left\{
\left(  1,2\right)  \right\}  }\left(  d_{1},d_{2}\right)  \right)
=\theta_{2}\circ\left(  \lambda_{\left(  d_{1},d_{2}\right)  :D^{2}\left\{
\left(  1,2\right)  \right\}  }\left(  d_{1},d_{2}\right)  \right)
\]
we have
\[
\left(  \alpha\underset{1}{\cdot}\theta_{1}\right)  \circ\left(
\lambda_{\left(  d_{1},d_{2}\right)  :D^{2}\left\{  \left(  1,2\right)
\right\}  }\left(  d_{1},d_{2}\right)  \right)  =\left(  \alpha\underset
{1}{\cdot}\theta_{2}\right)  \circ\left(  \lambda_{\left(  d_{1},d_{2}\right)
:D^{2}\left\{  \left(  1,2\right)  \right\}  }\left(  d_{1},d_{2}\right)
\right)
\]
and
\[
\left(  \alpha\underset{1}{\cdot}\theta_{1}\right)  \overset{\cdot}{-}\left(
\alpha\underset{1}{\cdot}\theta_{2}\right)  =\alpha\left(  \theta_{1}%
\overset{\cdot}{-}\theta_{2}\right)
\]

\end{proposition}

\begin{proof}
Since
\begin{align*}
&  \lambda_{\left(  d_{1},d_{2},d_{3}\right)  :D^{3}\left\{  \left(
1,3\right)  ,\left(  2,3\right)  \right\}  }\mathfrak{m}_{\left(
\alpha\underset{1}{\cdot}\theta_{1},\alpha\underset{1}{\cdot}\theta
_{2}\right)  }\left(  d_{1},d_{2},d_{3}\right) \\
&  =\lambda_{\left(  d_{1},d_{2},d_{3}\right)  :D^{3}\left\{  \left(
1,3\right)  ,\left(  2,3\right)  \right\}  }\mathfrak{m}_{\left(  \theta
_{1},\theta_{2}\right)  }\left(  \alpha d_{1},d_{2},\alpha d_{3}\right)
\end{align*}
we have
\begin{align*}
&  \lambda_{d:D}\left(  \alpha\left(  \theta_{1}\overset{\cdot}{-}\theta
_{2}\right)  \right)  \left(  d\right) \\
&  =\lambda_{d:D}\mathfrak{m}_{\left(  \theta_{1},\theta_{2}\right)  }\left(
0,0,\alpha d\right) \\
&  =\lambda_{d:D}\mathfrak{m}_{\left(  \alpha\underset{1}{\cdot}\theta
_{1},\alpha\underset{1}{\cdot}\theta_{2}\right)  }\left(  0,0,d\right) \\
&  =\lambda_{d:D}\left(  \left(  \alpha\underset{1}{\cdot}\theta_{1}\right)
\overset{\cdot}{-}\left(  \alpha\underset{1}{\cdot}\theta_{2}\right)  \right)
\left(  d\right)
\end{align*}

\end{proof}

\begin{lemma}
\label{l5.3}The diagram
\[%
\begin{array}
[c]{ccccc}
&  & D^{4}\left\{  \left(  1,3\right)  ,\left(  2,3\right)  ,\left(
1,4\right)  ,\left(  2,4\right)  ,\left(  3,4\right)  \right\}  &  & \\
& \nearrow & \uparrow & \nwarrow & \\
D^{2} &  & D^{2} &  & D^{2}\\
& \nwarrow & \uparrow & \nearrow & \\
&  & D^{2}\left\{  \left(  1,2\right)  \right\}  &  &
\end{array}
\]
is a quasi-colimit diagram, where the lower three arrows are all
\[
\lambda_{\left(  d_{1},d_{2}\right)  ::D^{2}\left\{  \left(  1,2\right)
\right\}  }\left(  d_{1},d_{2}\right)
\]
while the upper three arrows are
\begin{align*}
&  \lambda_{\left(  d_{1},d_{2}\right)  ::D^{2}}\left(  d_{1},d_{2},0,0\right)
\\
&  \lambda_{\left(  d_{1},d_{2}\right)  ::D^{2}}\left(  d_{1},d_{2},d_{1}%
d_{2},0\right) \\
&  \lambda_{\left(  d_{1},d_{2}\right)  ::D^{2}}\left(  d_{1},d_{2},d_{1}%
d_{2},d_{1}d_{2}\right)
\end{align*}
from left to right
\end{lemma}

\begin{proof}
By Axiom \ref{a1.1} we are sure that, given
\[
\gamma^{1},\gamma^{2},\gamma^{3}:D^{2}\rightarrow\mathbb{R}%
\]
there exist
\[
a^{1},a_{1}^{1},a_{2}^{1},a_{12}^{1},a^{2},a_{1}^{2},a_{2}^{2},a_{12}%
^{2},a^{3},a_{1}^{3},a_{2}^{3},a_{12}^{3}:\mathbb{R}%
\]
such that
\begin{align*}
&  \lambda_{\left(  d_{1},d_{2}\right)  :D^{2}}\gamma^{1}\left(  d_{1}%
,d_{2}\right) \\
&  =\lambda_{\left(  d_{1},d_{2}\right)  :D^{2}}\left(  a^{1}+a_{1}^{1}%
d_{1}+a_{2}^{1}d_{2}+a_{12}^{1}d_{1}d_{2}\right) \\
&  \lambda_{\left(  d_{1},d_{2}\right)  :D^{2}}\gamma^{2}\left(  d_{1}%
,d_{2}\right) \\
&  =\lambda_{\left(  d_{1},d_{2}\right)  :D^{2}}\left(  a^{2}+a_{1}^{2}%
d_{1}+a_{2}^{2}d_{2}+a_{12}^{2}d_{1}d_{2}\right) \\
&  \lambda_{\left(  d_{1},d_{2}\right)  :D^{2}}\gamma^{3}\left(  d_{1}%
,d_{2}\right) \\
&  =\lambda_{\left(  d_{1},d_{2}\right)  :D^{2}}\left(  a^{3}+a_{1}^{3}%
d_{1}+a_{2}^{3}d_{2}+a_{12}^{3}d_{1}d_{2}\right)
\end{align*}
If they satisfy
\begin{align*}
&  \gamma^{1}\circ\left(  \lambda_{\left(  d_{1},d_{2}\right)  ::D^{2}\left\{
\left(  1,2\right)  \right\}  }\left(  d_{1},d_{2}\right)  \right) \\
&  =\gamma^{2}\circ\left(  \lambda_{\left(  d_{1},d_{2}\right)  ::D^{2}%
\left\{  \left(  1,2\right)  \right\}  }\left(  d_{1},d_{2}\right)  \right) \\
&  =\gamma^{3}\circ\left(  \lambda_{\left(  d_{1},d_{2}\right)  ::D^{2}%
\left\{  \left(  1,2\right)  \right\}  }\left(  d_{1},d_{2}\right)  \right)
\end{align*}
then we have
\begin{align*}
a^{1}  &  =a^{2}=a^{3}\\
a_{1}^{1}  &  =a_{1}^{2}=a_{1}^{3}\\
a_{2}^{1}  &  =a_{2}^{2}=a_{3}^{3}%
\end{align*}
Therefore there exist
\[
b,b_{1},b_{2},b_{12},b_{3},b_{4}:\mathbb{R}%
\]
such that
\begin{align*}
&  \left(  \lambda_{\left(  d_{1},d_{2},d_{3},d_{4}\right)  :D^{4}\left\{
\left(  1,3\right)  ,\left(  2,3\right)  ,\left(  1,4\right)  ,\left(
2,4\right)  ,\left(  3,4\right)  \right\}  }b+b_{1}d_{1}+b_{2}d_{2}%
+b_{12}d_{1}d_{2}+b_{3}d_{3}+b_{4}d_{4}\right) \\
&  \circ\left(  \lambda_{\left(  d_{1},d_{2}\right)  :D^{2}}\gamma^{1}\left(
d_{1},d_{2}\right)  \right) \\
&  =\gamma^{1}\\
&  \left(  \lambda_{\left(  d_{1},d_{2},d_{3},d_{4}\right)  :D^{4}\left\{
\left(  1,3\right)  ,\left(  2,3\right)  ,\left(  1,4\right)  ,\left(
2,4\right)  ,\left(  3,4\right)  \right\}  }b+b_{1}d_{1}+b_{2}d_{2}%
+b_{12}d_{1}d_{2}+b_{3}d_{3}+b_{4}d_{4}\right) \\
&  \circ\lambda_{\left(  d_{1},d_{2}\right)  ::D^{2}}\left(  d_{1},d_{2}%
,d_{1}d_{2},0\right) \\
&  =\gamma^{2}\\
&  \left(  \lambda_{\left(  d_{1},d_{2},d_{3},d_{4}\right)  :D^{4}\left\{
\left(  1,3\right)  ,\left(  2,3\right)  ,\left(  1,4\right)  ,\left(
2,4\right)  ,\left(  3,4\right)  \right\}  }b+b_{1}d_{1}+b_{2}d_{2}%
+b_{12}d_{1}d_{2}+b_{3}d_{3}+b_{4}d_{4}\right) \\
&  \circ\lambda_{\left(  d_{1},d_{2}\right)  ::D^{2}}\left(  d_{1},d_{2}%
,d_{1}d_{2},d_{1}d_{2}\right) \\
&  =\gamma^{3}%
\end{align*}
This completes the proof.
\end{proof}

\begin{corollary}
\label{cl5.3}Let $M$\ be a microlinear set. Given $\theta_{1},\theta
_{2},\theta_{3}:D^{2}\rightarrow M$ with
\begin{align*}
&  \theta_{1}\circ\left(  \lambda_{\left(  d_{1},d_{2}\right)  :D^{2}\left\{
\left(  1,2\right)  \right\}  }\left(  d_{1},d_{2}\right)  \right) \\
&  =\theta_{2}\circ\left(  \lambda_{\left(  d_{1},d_{2}\right)  :D^{2}\left\{
\left(  1,2\right)  \right\}  }\left(  d_{1},d_{2}\right)  \right) \\
&  =\theta_{3}\circ\left(  \lambda_{\left(  d_{1},d_{2}\right)  :D^{2}\left\{
\left(  1,2\right)  \right\}  }\left(  d_{1},d_{2}\right)  \right)
\end{align*}
there exists
\[
\mathfrak{m}_{\left(  \theta_{1},\theta_{2},\theta_{3}\right)  }:D^{4}\left\{
\left(  1,3\right)  ,\left(  2,3\right)  ,\left(  1,4\right)  ,\left(
2,4\right)  ,\left(  3,4\right)  \right\}  \rightarrow M
\]
such that
\begin{align*}
\mathfrak{m}_{\left(  \theta_{1},\theta_{2},\theta_{3}\right)  }\circ\left(
\lambda_{\left(  d_{1},d_{2}\right)  ::D^{2}}\left(  d_{1},d_{2},d_{1}%
d_{2},d_{1}d_{2}\right)  \right)   &  =\theta_{1}\\
\mathfrak{m}_{\left(  \theta_{1},\theta_{2},\theta_{3}\right)  }\circ\left(
\lambda_{\left(  d_{1},d_{2}\right)  ::D^{2}}\left(  d_{1},d_{2},d_{1}%
d_{2},0\right)  \right)   &  =\theta_{2}\\
\mathfrak{m}_{\left(  \theta_{1},\theta_{2},\theta_{3}\right)  }\circ\left(
\lambda_{\left(  d_{1},d_{2}\right)  ::D^{2}}\left(  d_{1},d_{2},0,0\right)
\right)   &  =\theta_{3}%
\end{align*}

\end{corollary}

\begin{proposition}
\label{p5.4}(\underline{The primordial general Jacobi identity}) Let $M$\ be a
microlinear set. Given $\theta_{1},\theta_{2},\theta_{3}:D^{2}\rightarrow M$
with
\begin{align*}
&  \theta_{1}\circ\left(  \lambda_{\left(  d_{1},d_{2}\right)  :D^{2}\left\{
\left(  1,2\right)  \right\}  }\left(  d_{1},d_{2}\right)  \right) \\
&  =\theta_{2}\circ\left(  \lambda_{\left(  d_{1},d_{2}\right)  :D^{2}\left\{
\left(  1,2\right)  \right\}  }\left(  d_{1},d_{2}\right)  \right) \\
&  =\theta_{3}\circ\left(  \lambda_{\left(  d_{1},d_{2}\right)  :D^{2}\left\{
\left(  1,2\right)  \right\}  }\left(  d_{1},d_{2}\right)  \right)
\end{align*}
we have
\[
\left(  \theta_{1}\overset{\cdot}{-}\theta_{2}\right)  +\left(  \theta
_{2}\overset{\cdot}{-}\theta_{3}\right)  =\theta_{1}\overset{\cdot}{-}%
\theta_{3}%
\]
In particular, we have
\[
\left(  \theta_{1}\overset{\cdot}{-}\theta_{2}\right)  +\left(  \theta
_{2}\overset{\cdot}{-}\theta_{1}\right)  =0
\]

\end{proposition}

\begin{proof}
Since
\begin{align*}
&  \lambda_{\left(  d_{1},d_{2}\right)  :D^{2}\left\{  \left(  1,2\right)
\right\}  }\mathfrak{l}_{\left(  \theta_{1}\overset{\cdot}{-}\theta_{2}%
,\theta_{2}\overset{\cdot}{-}\theta_{3}\right)  }\left(  d_{1},d_{2}\right) \\
&  =\lambda_{\left(  d_{1},d_{2}\right)  :D^{2}\left\{  \left(  1,2\right)
\right\}  }\mathfrak{m}_{\left(  \theta_{1},\theta_{2},\theta_{3}\right)
}\left(  0,0,d_{1},d_{2}\right)
\end{align*}
and
\begin{align*}
&  \lambda_{\left(  d_{1},d_{2},d_{3}\right)  :D^{3}\left\{  \left(
1,3\right)  ,\left(  2,3\right)  \right\}  }\mathfrak{m}_{\left(  \theta
_{1},\theta_{3}\right)  }\left(  d_{1},d_{2},d_{3}\right) \\
&  =\lambda_{\left(  d_{1},d_{2},d_{3}\right)  :D^{3}\left\{  \left(
1,3\right)  ,\left(  2,3\right)  \right\}  }\mathfrak{m}_{\left(  \theta
_{1},\theta_{2},\theta_{3}\right)  }\left(  d_{1},d_{2},d_{3},d_{3}\right)
\end{align*}
we have
\begin{align*}
&  \lambda_{d:D}\left(  \left(  \theta_{1}\overset{\cdot}{-}\theta_{2}\right)
+\left(  \theta_{2}\overset{\cdot}{-}\theta_{3}\right)  \right)  \left(
d\right) \\
&  =\lambda_{d:D}\mathfrak{l}_{\left(  \theta_{1}\overset{\cdot}{-}\theta
_{2},\theta_{2}\overset{\cdot}{-}\theta_{3}\right)  }\left(  d,d\right) \\
&  =\lambda_{d:D}\mathfrak{m}_{\left(  \theta_{1},\theta_{2},\theta
_{3}\right)  }\left(  0,0,d,d\right) \\
&  =\lambda_{d:D}\mathfrak{m}_{\left(  \theta_{1},\theta_{3}\right)  }\left(
0,0,d\right) \\
&  =\lambda_{d:D}\left(  \theta_{1}\overset{\cdot}{-}\theta_{3}\right)
\left(  d\right)
\end{align*}
This completes the proof.
\end{proof}

Now we define relative strong differences.

\begin{definition}
Let $M$\ be a microlinear set. We give three definitions:

\begin{itemize}
\item Given $\theta_{1},\theta_{2}:D^{3}\rightarrow M$ with
\[
\theta_{1}\circ\left(  \lambda_{\left(  d_{1},d_{2},d_{3}\right)
:D^{3}\left\{  \left(  1,2\right)  \right\}  }\left(  d_{1},d_{2}%
,d_{3}\right)  \right)  =\theta_{2}\circ\left(  \lambda_{\left(  d_{1}%
,d_{2},d_{3}\right)  :D^{3}\left\{  \left(  1,2\right)  \right\}  }\left(
d_{1},d_{2},d_{3}\right)  \right)
\]
we define $\theta_{1}\underset{3}{\overset{\cdot}{-}}\theta_{2}:D^{2}%
\rightarrow M$\ to be
\begin{align*}
&  \theta_{1}\underset{3}{\overset{\cdot}{-}}\theta_{2}\\
&  :\equiv\lambda_{\left(  d_{1},d_{2}\right)  :D^{2}}\left(
\begin{array}
[c]{c}%
\left(
\begin{array}
[c]{c}%
\left(  \lambda_{\left(  d_{1},d_{2}\right)  :D^{2}}\lambda_{d_{3}:D}%
\theta_{1}\left(  d_{1},d_{2},d_{3}\right)  \right)  \overset{\cdot}{-}\\
\left(  \lambda_{\left(  d_{1},d_{2}\right)  :D^{2}}\lambda_{d_{3}:D}%
\theta_{2}\left(  d_{1},d_{2},d_{3}\right)  \right)
\end{array}
\right) \\
\circ\left(  \lambda_{\left(  d_{1},d_{2}\right)  :D^{2}}\left(  d_{2}%
,d_{1}\right)  \right)
\end{array}
\right)  \left(  d_{2}\right)  \left(  d_{1}\right)
\end{align*}

\item Given $\theta_{1},\theta_{2}:D^{3}\rightarrow M$ with
\[
\theta_{1}\circ\left(  \lambda_{\left(  d_{1},d_{2},d_{3}\right)
:D^{3}\left\{  \left(  1,3\right)  \right\}  }\left(  d_{1},d_{2}%
,d_{3}\right)  \right)  =\theta_{2}\circ\left(  \lambda_{\left(  d_{1}%
,d_{2},d_{3}\right)  :D^{3}\left\{  \left(  1,3\right)  \right\}  }\left(
d_{1},d_{2},d_{3}\right)  \right)
\]
we define $\theta_{1}\underset{2}{\overset{\cdot}{-}}\theta_{2}:D^{2}%
\rightarrow M$\ to be
\[
\theta_{1}\underset{2}{\overset{\cdot}{-}}\theta_{2}:\equiv\theta_{1}%
\circ\left(  \lambda_{\left(  d_{1},d_{2},d_{3}\right)  :D^{3}}\left(
d_{2},d_{3},d_{1}\right)  \right)  \underset{3}{\overset{\cdot}{-}}\theta
_{2}\circ\left(  \lambda_{\left(  d_{1},d_{2},d_{3}\right)  :D^{3}}\left(
d_{2},d_{3},d_{1}\right)  \right)
\]

\item Given $\theta_{1},\theta_{2}:D^{3}\rightarrow M$ with
\[
\theta_{1}\circ\left(  \lambda_{\left(  d_{1},d_{2},d_{3}\right)
:D^{3}\left\{  \left(  2,3\right)  \right\}  }\left(  d_{1},d_{2}%
,d_{3}\right)  \right)  =\theta_{2}\circ\left(  \lambda_{\left(  d_{1}%
,d_{2},d_{3}\right)  :D^{3}\left\{  \left(  2,3\right)  \right\}  }\left(
d_{1},d_{2},d_{3}\right)  \right)
\]
we define $\theta_{1}\underset{1}{\overset{\cdot}{-}}\theta_{2}:D^{2}%
\rightarrow M$\ to be
\[
\theta_{1}\underset{1}{\overset{\cdot}{-}}\theta_{2}:\equiv\theta_{1}%
\circ\left(  \lambda_{\left(  d_{1},d_{2},d_{3}\right)  :D^{3}}\left(
d_{3},d_{1},d_{2}\right)  \right)  \underset{3}{\overset{\cdot}{-}}\theta
_{2}\circ\left(  \lambda_{\left(  d_{1},d_{2},d_{3}\right)  :D^{3}}\left(
d_{3},d_{1},d_{2}\right)  \right)
\]

\end{itemize}
\end{definition}

\begin{lemma}
\label{l5.4}The diagram%
\begin{equation}%
\begin{array}
[c]{ccccc}
&  & D^{n+m_{1}+m_{2}}\left\{
\begin{array}
[c]{c}%
\left(  n+i,n+m_{1}+j\right)  \mid\\
1\leq i\leq m_{1},1\leq j\leq m_{2}%
\end{array}
\right\}  &  & \\
& \nearrow &  & \nwarrow & \\
D^{n+m_{1}+m_{2}}\left\{
\begin{array}
[c]{c}%
\left(  n+1\right)  ,...,\\
\left(  n+m_{1}\right)
\end{array}
\right\}  &  &  &  & D^{n+m_{1}+m_{2}}\left\{
\begin{array}
[c]{c}%
\left(  n+m_{1}+1\right)  ,...,\\
\left(  n+m_{1}+m_{2}\right)
\end{array}
\right\} \\
& \nwarrow &  & \nearrow & \\
&  & D^{n} &  &
\end{array}
\label{l5.4.1}%
\end{equation}
with the four arrows being the canonical injections is a quasi-colimit diagram.
\end{lemma}

\begin{corollary}
\label{cl5.4}Let $M$\ be a microlinear set. Given $\theta_{1},\theta
_{2}:D^{n+m_{1}+m_{2}}\rightarrow M$,
\begin{align*}
\theta_{1}  &  \mid D^{n+m_{1}+m_{2}}\left\{  \left(  n+i,n+m_{1}+j\right)
\mid1\leq i\leq m_{1},1\leq j\leq m_{2}\right\} \\
&  =\theta_{2}\mid D^{n+m_{1}+m_{2}}\left\{  \left(  n+i,n+m_{1}+j\right)
\mid1\leq i\leq m_{1},1\leq j\leq m_{2}\right\}
\end{align*}
obtains iff both%
\[
\theta_{1}\mid D^{n+m_{1}+m_{2}}\left\{  \left(  n+1\right)  ,...,\left(
n+m_{1}\right)  \right\}  =\theta_{2}\mid D^{n+m_{1}+m_{2}}\left\{  \left(
n+1\right)  ,...,\left(  n+m_{1}\right)  \right\}
\]
and%
\[
\theta_{1}\mid D^{n+m_{1}+m_{2}}\left\{  \left(  n+m_{1}+1\right)
,...,\left(  n+m_{1}+m_{2}\right)  \right\}  =\theta_{2}\mid D^{n+m_{1}+m_{2}%
}\left\{  \left(  n+m_{1}+1\right)  ,...,\left(  n+m_{1}+m_{2}\right)
\right\}
\]
obtain.
\end{corollary}

\begin{proposition}
\label{p5.5}Let $M$\ be a microlinear set and $\theta_{1},\theta_{2}%
,\theta_{3},\theta_{4}:D^{3}\rightarrow M$. Then we have the following:

\begin{itemize}
\item Let us suppose that the identities
\begin{align*}
\theta_{1}\circ\left(  \lambda_{\left(  d_{1},d_{2},d_{3}\right)
:D^{3}\left\{  \left(  2,3\right)  \right\}  }\left(  d_{1},d_{2}%
,d_{3}\right)  \right)   &  =\theta_{2}\circ\left(  \lambda_{\left(
d_{1},d_{2},d_{3}\right)  :D^{3}\left\{  \left(  2,3\right)  \right\}
}\left(  d_{1},d_{2},d_{3}\right)  \right) \\
\theta_{3}\circ\left(  \lambda_{\left(  d_{1},d_{2},d_{3}\right)
:D^{3}\left\{  \left(  2,3\right)  \right\}  }\left(  d_{1},d_{2}%
,d_{3}\right)  \right)   &  =\theta_{4}\circ\left(  \lambda_{\left(
d_{1},d_{2},d_{3}\right)  :D^{3}\left\{  \left(  2,3\right)  \right\}
}\left(  d_{1},d_{2},d_{3}\right)  \right)
\end{align*}
hold so that the strong differences
\begin{align*}
&  \theta_{1}\underset{1}{\overset{\cdot}{-}}\theta_{2}\\
&  \theta_{3}\underset{1}{\overset{\cdot}{-}}\theta_{4}%
\end{align*}
are to be defined. If the identies
\begin{align}
\theta_{1}\circ\left(  \lambda_{\left(  d_{1},d_{2},d_{3}\right)
:D^{3}\left\{  \left(  1,2\right)  ,\left(  1,3\right)  \right\}  }\left(
d_{1},d_{2},d_{3}\right)  \right)   &  =\theta_{3}\circ\left(  \lambda
_{\left(  d_{1},d_{2},d_{3}\right)  :D^{3}\left\{  \left(  1,2\right)
,\left(  1,3\right)  \right\}  }\left(  d_{1},d_{2},d_{3}\right)  \right)
\label{p5.5.0.1}\\
\theta_{2}\circ\left(  \lambda_{\left(  d_{1},d_{2},d_{3}\right)
:D^{3}\left\{  \left(  1,2\right)  ,\left(  1,3\right)  \right\}  }\left(
d_{1},d_{2},d_{3}\right)  \right)   &  =\theta_{4}\circ\left(  \lambda
_{\left(  d_{1},d_{2},d_{3}\right)  :D^{3}\left\{  \left(  1,2\right)
,\left(  1,3\right)  \right\}  }\left(  d_{1},d_{2},d_{3}\right)  \right)
\label{p5.5.0.2}%
\end{align}
obtain, then the identity
\[
\left(  \theta_{1}\underset{1}{\overset{\cdot}{-}}\theta_{2}\right)
\circ\left(  \lambda_{\left(  d_{1},d_{2}\right)  :D^{2}\left\{  \left(
1,2\right)  \right\}  }\left(  d_{1},d_{2}\right)  \right)  =\left(
\theta_{3}\underset{1}{\overset{\cdot}{-}}\theta_{4}\right)  \circ\left(
\lambda_{\left(  d_{1},d_{2}\right)  :D^{2}\left\{  \left(  1,2\right)
\right\}  }\left(  d_{1},d_{2}\right)  \right)
\]
obtains so that the strong difference
\[
\left(  \theta_{1}\underset{1}{\overset{\cdot}{-}}\theta_{2}\right)
\overset{\cdot}{-}\left(  \theta_{3}\underset{1}{\overset{\cdot}{-}}\theta
_{4}\right)
\]
is to be defined.

\item Let us suppose that the identities
\begin{align*}
\theta_{1}\circ\left(  \lambda_{\left(  d_{1},d_{2},d_{3}\right)
:D^{3}\left\{  \left(  1,3\right)  \right\}  }\left(  d_{1},d_{2}%
,d_{3}\right)  \right)   &  =\theta_{2}\circ\left(  \lambda_{\left(
d_{1},d_{2},d_{3}\right)  :D^{3}\left\{  \left(  1,3\right)  \right\}
}\left(  d_{1},d_{2},d_{3}\right)  \right) \\
\theta_{3}\circ\left(  \lambda_{\left(  d_{1},d_{2},d_{3}\right)
:D^{3}\left\{  \left(  1,3\right)  \right\}  }\left(  d_{1},d_{2}%
,d_{3}\right)  \right)   &  =\theta_{4}\circ\left(  \lambda_{\left(
d_{1},d_{2},d_{3}\right)  :D^{3}\left\{  \left(  1,3\right)  \right\}
}\left(  d_{1},d_{2},d_{3}\right)  \right)
\end{align*}
hold so that the strong differences
\begin{align*}
&  \theta_{1}\underset{2}{\overset{\cdot}{-}}\theta_{2}\\
&  \theta_{3}\underset{2}{\overset{\cdot}{-}}\theta_{4}%
\end{align*}
are to be defined. If the identies
\begin{align}
\theta_{1}\circ\left(  \lambda_{\left(  d_{1},d_{2},d_{3}\right)
:D^{3}\left\{  \left(  1,2\right)  ,\left(  2,3\right)  \right\}  }\left(
d_{1},d_{2},d_{3}\right)  \right)   &  =\theta_{3}\circ\left(  \lambda
_{\left(  d_{1},d_{2},d_{3}\right)  :D^{3}\left\{  \left(  1,2\right)
,\left(  2,3\right)  \right\}  }\left(  d_{1},d_{2},d_{3}\right)  \right)
\label{p5.5.0.3}\\
\theta_{2}\circ\left(  \lambda_{\left(  d_{1},d_{2},d_{3}\right)
:D^{3}\left\{  \left(  1,2\right)  ,\left(  2,3\right)  \right\}  }\left(
d_{1},d_{2},d_{3}\right)  \right)   &  =\theta_{4}\circ\left(  \lambda
_{\left(  d_{1},d_{2},d_{3}\right)  :D^{3}\left\{  \left(  1,2\right)
,\left(  2,3\right)  \right\}  }\left(  d_{1},d_{2},d_{3}\right)  \right)
\label{p5.5.0.4}%
\end{align}
obtain, then the identity
\[
\left(  \theta_{1}\underset{2}{\overset{\cdot}{-}}\theta_{2}\right)
\circ\left(  \lambda_{\left(  d_{1},d_{2}\right)  :D^{2}\left\{  \left(
1,2\right)  \right\}  }\left(  d_{1},d_{2}\right)  \right)  =\left(
\theta_{3}\underset{2}{\overset{\cdot}{-}}\theta_{4}\right)  \circ\left(
\lambda_{\left(  d_{1},d_{2}\right)  :D^{2}\left\{  \left(  1,2\right)
\right\}  }\left(  d_{1},d_{2}\right)  \right)
\]
obtains so that the strong difference
\[
\left(  \theta_{1}\underset{2}{\overset{\cdot}{-}}\theta_{2}\right)
\overset{\cdot}{-}\left(  \theta_{3}\underset{2}{\overset{\cdot}{-}}\theta
_{4}\right)
\]
is to be defined.

\item Let us suppose that the identities
\begin{align*}
\theta_{1}\circ\left(  \lambda_{\left(  d_{1},d_{2},d_{3}\right)
:D^{3}\left\{  \left(  1,2\right)  \right\}  }\left(  d_{1},d_{2}%
,d_{3}\right)  \right)   &  =\theta_{2}\circ\left(  \lambda_{\left(
d_{1},d_{2},d_{3}\right)  :D^{3}\left\{  \left(  1,2\right)  \right\}
}\left(  d_{1},d_{2},d_{3}\right)  \right) \\
\theta_{3}\circ\left(  \lambda_{\left(  d_{1},d_{2},d_{3}\right)
:D^{3}\left\{  \left(  1,2\right)  \right\}  }\left(  d_{1},d_{2}%
,d_{3}\right)  \right)   &  =\theta_{4}\circ\left(  \lambda_{\left(
d_{1},d_{2},d_{3}\right)  :D^{3}\left\{  \left(  1,2\right)  \right\}
}\left(  d_{1},d_{2},d_{3}\right)  \right)
\end{align*}
hold so that the strong differences
\begin{align*}
&  \theta_{1}\underset{3}{\overset{\cdot}{-}}\theta_{2}\\
&  \theta_{3}\underset{3}{\overset{\cdot}{-}}\theta_{4}%
\end{align*}
are to be defined. If the identies
\begin{align}
\theta_{1}\circ\left(  \lambda_{\left(  d_{1},d_{2},d_{3}\right)
:D^{3}\left\{  \left(  1,3\right)  ,\left(  2,3\right)  \right\}  }\left(
d_{1},d_{2},d_{3}\right)  \right)   &  =\theta_{3}\circ\left(  \lambda
_{\left(  d_{1},d_{2},d_{3}\right)  :D^{3}\left\{  \left(  1,3\right)
,\left(  2,3\right)  \right\}  }\left(  d_{1},d_{2},d_{3}\right)  \right)
\label{p5.5.0.5}\\
\theta_{2}\circ\left(  \lambda_{\left(  d_{1},d_{2},d_{3}\right)
:D^{3}\left\{  \left(  1,3\right)  ,\left(  2,3\right)  \right\}  }\left(
d_{1},d_{2},d_{3}\right)  \right)   &  =\theta_{4}\circ\left(  \lambda
_{\left(  d_{1},d_{2},d_{3}\right)  :D^{3}\left\{  \left(  1,3\right)
,\left(  2,3\right)  \right\}  }\left(  d_{1},d_{2},d_{3}\right)  \right)
\label{p5.5.0.6}%
\end{align}
obtain, then the identity
\[
\left(  \theta_{1}\underset{3}{\overset{\cdot}{-}}\theta_{2}\right)
\circ\left(  \lambda_{\left(  d_{1},d_{2}\right)  :D^{2}\left\{  \left(
1,2\right)  \right\}  }\left(  d_{1},d_{2}\right)  \right)  =\left(
\theta_{3}\underset{3}{\overset{\cdot}{-}}\theta_{4}\right)  \circ\left(
\lambda_{\left(  d_{1},d_{2}\right)  :D^{2}\left\{  \left(  1,2\right)
\right\}  }\left(  d_{1},d_{2}\right)  \right)
\]
obtains so that the strong difference
\[
\left(  \theta_{1}\underset{3}{\overset{\cdot}{-}}\theta_{2}\right)
\overset{\cdot}{-}\left(  \theta_{3}\underset{3}{\overset{\cdot}{-}}\theta
_{4}\right)
\]
is to be defined.
\end{itemize}
\end{proposition}

\begin{proof}
We deal only with the first statement, safely leaving the second and third
ones to the reader. We have to show that%
\[
\left(  \theta_{1}\underset{1}{\overset{\cdot}{-}}\theta_{2}\right)
\circ\left(  \lambda_{\left(  d_{1},d_{2}\right)  :D^{2}\left\{  \left(
1,2\right)  \right\}  }\left(  d_{1},d_{2}\right)  \right)  =\left(
\theta_{3}\underset{1}{\overset{\cdot}{-}}\theta_{4}\right)  \circ\left(
\lambda_{\left(  d_{1},d_{2}\right)  :D^{2}\left\{  \left(  1,2\right)
\right\}  }\left(  d_{1},d_{2}\right)  \right)
\]
which is, by dint of Corollary \ref{cl5.4} with respect to the quasi-colimit
diagram%
\[%
\begin{array}
[c]{ccccc}
&  & D^{2}\left\{  \left(  1,2\right)  \right\}  &  & \\
& \nearrow &  & \nwarrow & \\
D^{2}\left\{  \left(  1\right)  \right\}  &  &  &  & D^{2}\left\{  \left(
2\right)  \right\} \\
& \nwarrow &  & \nearrow & \\
&  & D^{2}\left\{  \left(  1\right)  ,\left(  2\right)  \right\}  &  &
\end{array}
\]
with the four arrows being the canonical injections (Lemma \ref{l5.4} with
$n=0$, $m_{1}=1$ and $m_{2}=1$), tantamount to showing that%
\begin{align}
\left(  \theta_{1}\underset{1}{\overset{\cdot}{-}}\theta_{2}\right)   &  \mid
D^{2}\left\{  \left(  1\right)  \right\}  =\left(  \theta_{3}\underset
{1}{\overset{\cdot}{-}}\theta_{4}\right)  \mid D^{2}\left\{  \left(  1\right)
\right\} \label{p5.5.1}\\
\left(  \theta_{1}\underset{1}{\overset{\cdot}{-}}\theta_{2}\right)   &  \mid
D^{2}\left\{  \left(  2\right)  \right\}  =\left(  \theta_{3}\underset
{1}{\overset{\cdot}{-}}\theta_{4}\right)  \mid D^{2}\left\{  \left(  2\right)
\right\}  \label{p5.5.2}%
\end{align}
Due to the quasi-colimit diagram
\[%
\begin{array}
[c]{ccccc}
&  & D^{3}\left\{  \left(  1,2\right)  ,\left(  1,3\right)  \right\}  &  & \\
& \nearrow &  & \nwarrow & \\
D^{3}\left\{  \left(  \left(  1\right)  \right)  \right\}  &  &  &  &
D^{3}\left\{  \left(  2\right)  ,\left(  3\right)  \right\} \\
& \nwarrow &  & \nearrow & \\
&  & D^{3}\left\{  \left(  1\right)  ,\left(  2\right)  ,\left(  3\right)
\right\}  &  &
\end{array}
\]
with the four arrows being the canonical injections (Lemma \ref{l5.4} with
$n=0$, $m_{1}=1$ and $m_{2}=2$), the condition (\ref{p5.5.0.1}) is equivalent
by dint of Corollary \ref{cl5.4} to the conditions
\begin{align}
\theta_{1}  &  \mid D^{3}\left\{  \left(  1\right)  \right\}  =\theta_{3}\mid
D^{3}\left\{  \left(  1\right)  \right\} \label{p5.5.3}\\
\theta_{1}  &  \mid D^{3}\left\{  \left(  2\right)  ,\left(  3\right)
\right\}  =\theta_{3}\mid D^{3}\left\{  \left(  2\right)  ,\left(  3\right)
\right\}  \label{p5.5.4}%
\end{align}
while the condition (\ref{p5.5.0.2}) is equivalent to the conditions
\begin{align}
\theta_{2}  &  \mid D^{3}\left\{  \left(  1\right)  \right\}  =\theta_{4}\mid
D^{3}\left\{  \left(  1\right)  \right\} \label{p5.5.5}\\
\theta_{2}  &  \mid D^{3}\left\{  \left(  2\right)  ,\left(  3\right)
\right\}  =\theta_{4}\mid D^{3}\left\{  \left(  2\right)  ,\left(  3\right)
\right\}  \label{p5.5.6}%
\end{align}
In order to show that (\ref{p5.5.1}) obtains, we note that the quasi-colimit
diagram (cf. Lemma 2.1 in \cite{ni1})%
\[%
\begin{array}
[c]{ccccc}
&  & D^{4}\left\{  \left(  2,4\right)  ,\left(  3,4\right)  \right\}  &  & \\
& \nearrow &  & \nwarrow & \\
D^{3} &  &  &  & D^{3}\\
& \nwarrow &  & \nearrow & \\
&  & D^{3}\left\{  \left(  2,3\right)  \right\}  &  &
\end{array}
\]
with the upper arrows being%
\begin{align*}
&  \lambda_{\left(  d_{1},d_{2},d_{3}\right)  :D^{3}}\left(  d_{1},d_{2}%
,d_{3},0\right) \\
&  \lambda_{\left(  d_{1},d_{2},d_{3}\right)  :D^{3}}\left(  d_{1},d_{2}%
,d_{3},d_{2}d_{3}\right)
\end{align*}
from left to right and the lower arrow being the canonical injections is to be
restricted to the quasi-colimit diagram
\[%
\begin{array}
[c]{ccccc}
&  & D^{4}\left\{  \left(  1\right)  ,\left(  2,4\right)  ,\left(  3,4\right)
\right\}  &  & \\
& \nearrow &  & \nwarrow & \\
D^{3}\left\{  \left(  1\right)  \right\}  &  &  &  & D^{3}\left\{  \left(
1\right)  \right\} \\
& \nwarrow &  & \nearrow & \\
&  & D^{3}\left\{  \left(  1\right)  ,\left(  2,3\right)  \right\}  &  &
\end{array}
\]
so that the conditions (\ref{p5,5,3}) and (\ref{p5.5.5}) imply (\ref{p5.5.1}).
It is easy to see that
\begin{align*}
&  \left(  \left(  \theta_{1}\overset{\cdot}{\underset{1}{-}}\theta
_{2}\right)  \mid D^{2}\left\{  \left(  2\right)  \right\}  \right)
\circ\left(  \lambda_{d:D}\left(  d,0\right)  \right) \\
&  =\left(  \theta_{1}\mid D^{3}\left\{  \left(  2\right)  ,\left(  3\right)
\right\}  \right)  \circ\left(  \lambda_{d:D}\left(  d,0,0\right)  \right)
=\left(  \theta_{2}\mid D^{3}\left\{  \left(  2\right)  ,\left(  3\right)
\right\}  \right)  \circ\left(  \lambda_{d:D}\left(  d,0,0\right)  \right) \\
&  \left(  \left(  \theta_{3}\overset{\cdot}{\underset{1}{-}}\theta
_{4}\right)  \mid D^{2}\left\{  \left(  2\right)  \right\}  \right)
\circ\left(  \lambda_{d:D}\left(  d,0\right)  \right) \\
&  =\left(  \theta_{3}\mid D^{3}\left\{  \left(  2\right)  ,\left(  3\right)
\right\}  \right)  \circ\left(  \lambda_{d:D}\left(  d,0,0\right)  \right)
=\left(  \theta_{4}\mid D^{3}\left\{  \left(  2\right)  ,\left(  3\right)
\right\}  \right)  \circ\left(  \lambda_{d:D}\left(  d,0,0\right)  \right)
\end{align*}
obtain so that (\ref{p5.5.4}) and (\ref{p5.5.6}) imply (\ref{p5.5.2}). This
completes the proof.
\end{proof}

\begin{lemma}
\label{l5.5}The diagram consisting of
\begin{align*}
\mathcal{G}  &  :\equiv D^{8}\left\{
\begin{array}
[c]{c}%
\left(  2,4\right)  ,\left(  3,4\right)  ,\left(  1,5\right)  ,\left(
3,5\right)  ,\left(  1,6\right)  ,\left(  2,6\right)  ,\left(  4,5\right)
,\left(  4,6\right)  ,\left(  5,6\right)  ,\\
\left(  1,7\right)  ,\left(  2,7\right)  ,\left(  3,7\right)  ,\left(
4,7\right)  ,\left(  5,7\right)  ,\left(  6,7\right)  ,\left(  7,8\right)  ,\\
\left(  1,8\right)  ,\left(  2,8\right)  ,\left(  3,8\right)  ,\left(
4,8\right)  ,\left(  5,8\right)  ,\left(  6,8\right)
\end{array}
\right\} \\
\mathcal{H}_{123}  &  :\equiv D^{3},\mathcal{H}_{132}:\equiv D^{3}%
,\mathcal{H}_{213}:\equiv D^{3},\mathcal{H}_{231}:\equiv D^{3},\mathcal{H}%
_{312}:\equiv D^{3},\mathcal{H}_{321}:\equiv D^{3}\\
\mathcal{K}_{123,132}^{1}  &  :\equiv D^{3}\left\{  \left(  2,3\right)
\right\}  ,\mathcal{K}_{231,321}^{1}:\equiv D^{3}\left\{  \left(  2,3\right)
\right\}  ,\mathcal{K}_{231,213}^{2}:\equiv D^{3}\left\{  \left(  1,3\right)
\right\}  ,\\
\mathcal{K}_{312,132}^{2}  &  :\equiv D^{3}\left\{  \left(  1,3\right)
\right\}  ,\mathcal{K}_{312,321}^{3}:\equiv D^{3}\left\{  \left(  1,2\right)
\right\}  ,\mathcal{K}_{123,213}^{3}:\equiv D^{3}\left\{  \left(  1,2\right)
\right\} \\
f_{123}  &  :\equiv\lambda_{\left(  d_{1},d_{2},d_{3}\right)  :D^{3}}\left(
d_{1},d_{2},d_{3},0,0,0,0,0\right)  :\mathcal{H}_{123}\rightarrow\mathcal{G}\\
f_{132}  &  :\equiv\lambda_{\left(  d_{1},d_{2},d_{3}\right)  :D^{3}}\left(
d_{1},d_{2},d_{3},d_{2}d_{3},0,0,0,0\right)  :\mathcal{H}_{132}\rightarrow
\mathcal{G}\\
f_{213}  &  :\equiv\lambda_{\left(  d_{1},d_{2},d_{3}\right)  :D^{3}}\left(
d_{1},d_{2},d_{3},0,0,d_{1}d_{2},0,0\right)  :\mathcal{H}_{213}\rightarrow
\mathcal{G}\\
f_{231}  &  :\equiv\lambda_{\left(  d_{1},d_{2},d_{3}\right)  :D^{3}}\left(
d_{1},d_{2},d_{3},0,d_{1}d_{3},d_{1}d_{2},0,0\right)  :\mathcal{H}%
_{231}\rightarrow\mathcal{G}\\
f_{312}  &  :\equiv\lambda_{\left(  d_{1},d_{2},d_{3}\right)  :D^{3}}\left(
d_{1},d_{2},d_{3},d_{2}d_{3},d_{1}d_{3},0,d_{1}d_{2}d_{3},0\right)
:\mathcal{H}_{312}\rightarrow\mathcal{G}\\
f_{321}  &  :\equiv\lambda_{\left(  d_{1},d_{2},d_{3}\right)  :D^{3}}\left(
d_{1},d_{2},d_{3},d_{2}d_{3},d_{1}d_{3},d_{1}d_{2},0,d_{1}d_{2}d_{3}\right)
:\mathcal{H}_{321}\rightarrow\mathcal{G}\\
h_{123}^{1}  &  :\equiv\lambda_{\left(  d_{1},d_{2},d_{3}\right)
:D^{3}\left\{  \left(  2,3\right)  \right\}  }\left(  d_{1},d_{2}%
,d_{3}\right)  :\mathcal{K}_{123,132}^{1}\rightarrow\mathcal{H}_{123}\\
h_{132}^{1}  &  :\equiv\lambda_{\left(  d_{1},d_{2},d_{3}\right)
:D^{3}\left\{  \left(  2,3\right)  \right\}  }\left(  d_{1},d_{2}%
,d_{3}\right)  :\mathcal{K}_{123,132}^{1}\rightarrow\mathcal{H}_{132}\\
h_{231}^{1}  &  :\equiv\lambda_{\left(  d_{1},d_{2},d_{3}\right)
:D^{3}\left\{  \left(  2,3\right)  \right\}  }\left(  d_{1},d_{2}%
,d_{3}\right)  :\mathcal{K}_{231,321}^{1}\rightarrow\mathcal{H}_{231}\\
h_{321}^{1}  &  :\equiv\lambda_{\left(  d_{1},d_{2},d_{3}\right)
:D^{3}\left\{  \left(  2,3\right)  \right\}  }\left(  d_{1},d_{2}%
,d_{3}\right)  :\mathcal{K}_{231,321}^{1}\rightarrow\mathcal{H}_{321}\\
h_{231}^{2}  &  :\equiv\lambda_{\left(  d_{1},d_{2},d_{3}\right)
:D^{3}\left\{  \left(  1,3\right)  \right\}  }\left(  d_{1},d_{2}%
,d_{3}\right)  :\mathcal{K}_{231,213}^{2}\rightarrow\mathcal{H}_{231}\\
h_{213}^{2}  &  :\equiv\lambda_{\left(  d_{1},d_{2},d_{3}\right)
:D^{3}\left\{  \left(  1,3\right)  \right\}  }\left(  d_{1},d_{2}%
,d_{3}\right)  :\mathcal{K}_{231,213}^{2}\rightarrow\mathcal{H}_{213}\\
h_{312}^{2}  &  :\equiv\lambda_{\left(  d_{1},d_{2},d_{3}\right)
:D^{3}\left\{  \left(  1,3\right)  \right\}  }\left(  d_{1},d_{2}%
,d_{3}\right)  :\mathcal{K}_{312,132}^{2}\rightarrow\mathcal{H}_{312}\\
h_{132}^{2}  &  :\equiv\lambda_{\left(  d_{1},d_{2},d_{3}\right)
:D^{3}\left\{  \left(  1,3\right)  \right\}  }\left(  d_{1},d_{2}%
,d_{3}\right)  :\mathcal{K}_{312,132}^{2}\rightarrow\mathcal{H}_{132}\\
h_{312}^{3}  &  :\equiv\lambda_{\left(  d_{1},d_{2},d_{3}\right)
:D^{3}\left\{  \left(  1,2\right)  \right\}  }\left(  d_{1},d_{2}%
,d_{3}\right)  :\mathcal{K}_{312,321}^{3}\rightarrow\mathcal{H}_{312}\\
h_{321}^{3}  &  :\equiv\lambda_{\left(  d_{1},d_{2},d_{3}\right)
:D^{3}\left\{  \left(  1,2\right)  \right\}  }\left(  d_{1},d_{2}%
,d_{3}\right)  :\mathcal{K}_{312,321}^{3}\rightarrow\mathcal{H}_{321}\\
h_{123}^{3}  &  :\equiv\lambda_{\left(  d_{1},d_{2},d_{3}\right)
:D^{3}\left\{  \left(  1,2\right)  \right\}  }\left(  d_{1},d_{2}%
,d_{3}\right)  :\mathcal{K}_{123,213}^{3}\rightarrow\mathcal{H}_{123}\\
h_{213}^{3}  &  :\equiv\lambda_{\left(  d_{1},d_{2},d_{3}\right)
:D^{3}\left\{  \left(  1,2\right)  \right\}  }\left(  d_{1},d_{2}%
,d_{3}\right)  :\mathcal{K}_{123,213}^{3}\rightarrow\mathcal{H}_{213}%
\end{align*}
is a quasi-colimit diagram.
\end{lemma}

\begin{proof}
By Axiom \ref{a1.1} we are sure that, given
\[
\gamma^{123},\gamma^{132},\gamma^{213},\gamma^{231},\gamma^{312},\gamma
^{321}:D^{3}\rightarrow\mathbb{R}%
\]
there exist
\begin{align*}
a^{123},a_{1}^{123},a_{2}^{123},a_{3}^{123},a_{12}^{123},a_{13}^{123}%
,a_{23}^{123},a_{123}^{123}  &  :\mathbb{R}\\
a^{132},a_{1}^{132},a_{2}^{132},a_{3}^{132},a_{12}^{132},a_{13}^{132}%
,a_{23}^{132},a_{123}^{132}  &  :\mathbb{R}\\
a^{213},a_{1}^{213},a_{2}^{213},a_{3}^{213},a_{12}^{213},a_{13}^{213}%
,a_{23}^{213},a_{123}^{213}  &  :\mathbb{R}\\
a^{231},a_{1}^{231},a_{2}^{231},a_{3}^{231},a_{12}^{231},a_{13}^{231}%
,a_{23}^{231},a_{123}^{231}  &  :\mathbb{R}\\
a^{312},a_{1}^{312},a_{2}^{312},a_{3}^{312},a_{12}^{312},a_{13}^{312}%
,a_{23}^{312},a_{123}^{312}  &  :\mathbb{R}\\
a^{321},a_{1}^{321},a_{2}^{321},a_{3}^{321},a_{12}^{321},a_{13}^{321}%
,a_{23}^{321},a_{123}^{321}  &  :\mathbb{R}%
\end{align*}
such that
\begin{align*}
&  \lambda_{\left(  d_{1},d_{2},d_{3}\right)  :D^{3}}\gamma^{123}\left(
d_{1},d_{2},d_{3}\right) \\
&  =\lambda_{\left(  d_{1},d_{2},d_{3}\right)  :D^{3}}\left(
\begin{array}
[c]{c}%
a^{123}+a_{1}^{123}d_{1}+a_{2}^{123}d_{2}+a_{3}^{123}d_{3}+\\
a_{12}^{123}d_{1}d_{2}+a_{13}^{123}d_{1}d_{3}+a_{23}^{123}d_{2}d_{3}%
+a_{123}^{123}d_{1}d_{2}d_{3}%
\end{array}
\right) \\
&  \lambda_{\left(  d_{1},d_{2},d_{3}\right)  :D^{3}}\gamma^{132}\left(
d_{1},d_{2},d_{3}\right) \\
&  =\lambda_{\left(  d_{1},d_{2},d_{3}\right)  :D^{3}}\left(
\begin{array}
[c]{c}%
a^{132}+a_{1}^{132}d_{1}+a_{2}^{132}d_{2}+a_{3}^{132}d_{3}+\\
a_{12}^{132}d_{1}d_{2}+a_{13}^{132}d_{1}d_{3}+a_{23}^{132}d_{2}d_{3}%
+a_{123}^{132}d_{1}d_{2}d_{3}%
\end{array}
\right) \\
&  \lambda_{\left(  d_{1},d_{2},d_{3}\right)  :D^{3}}\gamma^{213}\left(
d_{1},d_{2},d_{3}\right) \\
&  =\lambda_{\left(  d_{1},d_{2},d_{3}\right)  :D^{3}}\left(
\begin{array}
[c]{c}%
a^{213}+a_{1}^{213}d_{1}+a_{2}^{213}d_{2}+a_{3}^{213}d_{3}+\\
a_{12}^{213}d_{1}d_{2}+a_{13}^{213}d_{1}d_{3}+a_{23}^{213}d_{2}d_{3}%
+a_{123}^{213}d_{1}d_{2}d_{3}%
\end{array}
\right) \\
&  \lambda_{\left(  d_{1},d_{2},d_{3}\right)  :D^{3}}\gamma^{231}\left(
d_{1},d_{2},d_{3}\right) \\
&  =\lambda_{\left(  d_{1},d_{2},d_{3}\right)  :D^{3}}\left(
\begin{array}
[c]{c}%
a^{231}+a_{1}^{231}d_{1}+a_{2}^{231}d_{2}+a_{3}^{231}d_{3}+\\
a_{12}^{231}d_{1}d_{2}+a_{13}^{231}d_{1}d_{3}+a_{23}^{231}d_{2}d_{3}%
+a_{123}^{231}d_{1}d_{2}d_{3}%
\end{array}
\right) \\
&  \lambda_{\left(  d_{1},d_{2},d_{3}\right)  :D^{3}}\gamma^{312}\left(
d_{1},d_{2},d_{3}\right) \\
&  =\lambda_{\left(  d_{1},d_{2},d_{3}\right)  :D^{3}}\left(
\begin{array}
[c]{c}%
a^{312}+a_{1}^{312}d_{1}+a_{2}^{312}d_{2}+a_{3}^{312}d_{3}+\\
a_{12}^{312}d_{1}d_{2}+a_{13}^{312}d_{1}d_{3}+a_{23}^{312}d_{2}d_{3}%
+a_{123}^{312}d_{1}d_{2}d_{3}%
\end{array}
\right) \\
&  \lambda_{\left(  d_{1},d_{2},d_{3}\right)  :D^{3}}\gamma^{321}\left(
d_{1},d_{2},d_{3}\right) \\
&  =\lambda_{\left(  d_{1},d_{2},d_{3}\right)  :D^{3}}\left(
\begin{array}
[c]{c}%
a^{321}+a_{1}^{321}d_{1}+a_{2}^{321}d_{2}+a_{3}^{321}d_{3}+\\
a_{12}^{321}d_{1}d_{2}+a_{13}^{321}d_{1}d_{3}+a_{23}^{321}d_{2}d_{3}%
+a_{123}^{321}d_{1}d_{2}d_{3}%
\end{array}
\right)
\end{align*}
If it holds that
\begin{align*}
\gamma^{123}\circ h_{123}^{1}  &  =\gamma^{132}\circ h_{132}^{1}\\
\gamma^{231}\circ h_{231}^{1}  &  =\gamma^{321}\circ h_{321}^{1}\\
\gamma^{231}\circ h_{231}^{2}  &  =\gamma^{213}\circ h_{213}^{2}\\
\gamma^{312}\circ h_{312}^{2}  &  =\gamma^{132}\circ h_{132}^{2}\\
\gamma^{312}\circ h_{312}^{3}  &  =\gamma^{321}\circ h_{321}^{3}\\
\gamma^{123}\circ h_{123}^{3}  &  =\gamma^{213}\circ h_{213}^{3}%
\end{align*}
then we have
\begin{align*}
a^{123}  &  =a^{132},a_{1}^{123}=a_{1}^{132},a_{2}^{123}=a_{2}^{132}%
,a_{3}^{123}=a_{3}^{132},a_{12}^{123}=a_{12}^{132},a_{13}^{123}=a_{13}^{132}\\
a^{231}  &  =a^{321},a_{1}^{231}=a_{1}^{321},a_{2}^{231}=a_{2}^{321}%
,a_{3}^{231}=a_{3}^{321},a_{12}^{231}=a_{12}^{321},a_{13}^{231}=a_{13}^{321}\\
a^{231}  &  =a^{213},a_{1}^{231}=a_{1}^{213},a_{2}^{231}=a_{2}^{213}%
,a_{3}^{231}=a_{3}^{213},a_{12}^{231}=a_{12}^{213},a_{23}^{231}=a_{23}^{213}\\
a^{312}  &  =a^{132},a_{1}^{312}=a_{1}^{132},a_{2}^{312}=a_{2}^{132}%
,a_{3}^{312}=a_{3}^{132},a_{12}^{312}=a_{12}^{132},a_{23}^{312}=a_{23}^{132}\\
a^{312}  &  =a^{321},a_{1}^{312}=a_{1}^{321},a_{2}^{312}=a_{2}^{321}%
,a_{3}^{312}=a_{3}^{321},a_{13}^{312}=a_{13}^{321},a_{23}^{312}=a_{23}^{321}\\
a^{123}  &  =a^{213},a_{1}^{123}=a_{1}^{213},a_{2}^{123}=a_{2}^{213}%
,a_{3}^{123}=a_{3}^{213},a_{13}^{123}=a_{13}^{213},a_{23}^{123}=a_{23}^{213}%
\end{align*}%
\begin{align*}
a^{123}  &  =a^{231},a_{1}^{123}=a_{1}^{231},a_{2}^{123}=a_{2}^{231}%
,a_{3}^{123}=a_{3}^{231},a_{23}^{123}=a_{23}^{231}\\
a^{132}  &  =a^{321},a_{1}^{132}=a_{1}^{321},a_{2}^{132}=a_{2}^{321}%
,a_{3}^{132}=a_{3}^{321},a_{23}^{132}=a_{23}^{321}\\
a^{231}  &  =a^{312},a_{1}^{231}=a_{1}^{312},a_{2}^{231}=a_{2}^{312}%
,a_{3}^{231}=a_{3}^{312},a_{13}^{231}=a_{13}^{312}\\
a^{213}  &  =a^{132},a_{1}^{213}=a_{1}^{132},a_{2}^{213}=a_{2}^{132}%
,a_{3}^{213}=a_{3}^{132},a_{13}^{213}=a_{13}^{132}\\
a^{312}  &  =a^{123},a_{1}^{312}=a_{1}^{123},a_{2}^{312}=a_{2}^{123}%
,a_{3}^{312}=a_{3}^{123},a_{12}^{312}=a_{12}^{123}\\
a^{321}  &  =a^{213},a_{1}^{321}=a_{1}^{213},a_{2}^{321}=a_{2}^{213}%
,a_{3}^{321}=a_{3}^{213},a_{12}^{321}=a_{12}^{213}%
\end{align*}
which is tantamout to
\begin{align*}
a^{123}  &  =a^{132}=a^{213}=a^{231}=a^{312}=a^{321}\\
a_{1}^{123}  &  =a_{1}^{132}=a_{1}^{213}=a_{1}^{231}=a_{1}^{312}=a_{1}^{321}\\
a_{2}^{123}  &  =a_{2}^{132}=a_{2}^{213}=a_{2}^{231}=a_{2}^{312}=a_{2}^{321}\\
a_{3}^{123}  &  =a_{3}^{132}=a_{3}^{213}=a_{3}^{231}=a_{3}^{312}=a_{3}^{321}\\
a_{12}^{123}  &  =a_{12}^{132}=a_{12}^{312}\text{ and }a_{12}^{213}%
=a_{12}^{231}=a_{12}^{321}\\
a_{13}^{123}  &  =a_{13}^{132}=a_{13}^{213}\text{ and }a_{13}^{231}%
=a_{13}^{312}=a_{13}^{321}\\
a_{23}^{123}  &  =a_{23}^{213}=a_{23}^{231}\text{ and }a_{23}^{132}%
=a_{23}^{312}=a_{23}^{321}%
\end{align*}
This means that there exist
\[
b,b_{1},b_{2},b_{3},b_{4},b_{5},b_{6},b_{7},b_{8},b_{12},b_{13},b_{23}%
,b_{123},b_{14},b_{25},b_{36}:\mathbb{R}%
\]
such that
\begin{align*}
&  \left(  \lambda_{\left(  d_{1},d_{2},d_{3},d_{4},d_{5},d_{6},d_{7}%
,d_{8}\right)  :G}\left(
\begin{array}
[c]{c}%
b+b_{1}d_{1}+b_{2}d_{2}+b_{3}d_{3}+b_{4}d_{4}+b_{5}d_{5}+\\
b_{6}d_{6}+b_{7}d_{7}+b_{8}d_{8}+\\
b_{12}d_{1}d_{2}+b_{13}d_{1}d_{3}+b_{23}d_{2}d_{3}+b_{123}d_{1}d_{2}d_{3}+\\
b_{14}d_{1}d_{4}+b_{25}d_{2}d_{5}+b_{36}d_{3}d_{6}%
\end{array}
\right)  \right)  \circ f_{123}\\
&  =\gamma^{123}\\
&  \left(  \lambda_{\left(  d_{1},d_{2},d_{3},d_{4},d_{5},d_{6},d_{7}%
,d_{8}\right)  :G}\left(
\begin{array}
[c]{c}%
b+b_{1}d_{1}+b_{2}d_{2}+b_{3}d_{3}+b_{4}d_{4}+b_{5}d_{5}+\\
b_{6}d_{6}+b_{7}d_{7}+b_{8}d_{8}+\\
b_{12}d_{1}d_{2}+b_{13}d_{1}d_{3}+b_{23}d_{2}d_{3}+b_{123}d_{1}d_{2}d_{3}+\\
b_{14}d_{1}d_{4}+b_{25}d_{2}d_{5}+b_{36}d_{3}d_{6}%
\end{array}
\right)  \right)  \circ f_{132}\\
&  =\gamma^{132}\\
&  \left(  \lambda_{\left(  d_{1},d_{2},d_{3},d_{4},d_{5},d_{6},d_{7}%
,d_{8}\right)  :G}\left(
\begin{array}
[c]{c}%
b+b_{1}d_{1}+b_{2}d_{2}+b_{3}d_{3}+b_{4}d_{4}+b_{5}d_{5}+\\
b_{6}d_{6}+b_{7}d_{7}+b_{8}d_{8}+\\
b_{12}d_{1}d_{2}+b_{13}d_{1}d_{3}+b_{23}d_{2}d_{3}+b_{123}d_{1}d_{2}d_{3}+\\
b_{14}d_{1}d_{4}+b_{25}d_{2}d_{5}+b_{36}d_{3}d_{6}%
\end{array}
\right)  \right)  \circ f_{213}\\
&  =\gamma^{213}\\
&  \left(  \lambda_{\left(  d_{1},d_{2},d_{3},d_{4},d_{5},d_{6},d_{7}%
,d_{8}\right)  :G}\left(
\begin{array}
[c]{c}%
b+b_{1}d_{1}+b_{2}d_{2}+b_{3}d_{3}+b_{4}d_{4}+b_{5}d_{5}+\\
b_{6}d_{6}+b_{7}d_{7}+b_{8}d_{8}+\\
b_{12}d_{1}d_{2}+b_{13}d_{1}d_{3}+b_{23}d_{2}d_{3}+b_{123}d_{1}d_{2}d_{3}+\\
b_{14}d_{1}d_{4}+b_{25}d_{2}d_{5}+b_{36}d_{3}d_{6}%
\end{array}
\right)  \right)  \circ f_{231}\\
&  =\gamma^{231}\\
&  \left(  \lambda_{\left(  d_{1},d_{2},d_{3},d_{4},d_{5},d_{6},d_{7}%
,d_{8}\right)  :G}\left(
\begin{array}
[c]{c}%
b+b_{1}d_{1}+b_{2}d_{2}+b_{3}d_{3}+b_{4}d_{4}+b_{5}d_{5}+\\
b_{6}d_{6}+b_{7}d_{7}+b_{8}d_{8}+\\
b_{12}d_{1}d_{2}+b_{13}d_{1}d_{3}+b_{23}d_{2}d_{3}+b_{123}d_{1}d_{2}d_{3}+\\
b_{14}d_{1}d_{4}+b_{25}d_{2}d_{5}+b_{36}d_{3}d_{6}%
\end{array}
\right)  \right)  \circ f_{312}\\
&  =\gamma^{312}\\
&  \left(  \lambda_{\left(  d_{1},d_{2},d_{3},d_{4},d_{5},d_{6},d_{7}%
,d_{8}\right)  :G}\left(
\begin{array}
[c]{c}%
b+b_{1}d_{1}+b_{2}d_{2}+b_{3}d_{3}+b_{4}d_{4}+b_{5}d_{5}+\\
b_{6}d_{6}+b_{7}d_{7}+b_{8}d_{8}+\\
b_{12}d_{1}d_{2}+b_{13}d_{1}d_{3}+b_{23}d_{2}d_{3}+b_{123}d_{1}d_{2}d_{3}+\\
b_{14}d_{1}d_{4}+b_{25}d_{2}d_{5}+b_{36}d_{3}d_{6}%
\end{array}
\right)  \right)  \circ f_{321}\\
&  =\gamma^{321}%
\end{align*}
This completes the proof.
\end{proof}

\begin{corollary}
\label{cl5.5}Let $M$\ be a microlinear set. Given
\[
\theta_{123},\theta_{132},\theta_{213},\theta_{231},\theta_{312},\theta
_{321}:D^{3}\rightarrow M
\]
with
\begin{align*}
\theta_{123}\circ\left(  \lambda_{\left(  d_{1},d_{2},d_{3}\right)
:D^{3}\left\{  \left(  2,3\right)  \right\}  }\left(  d_{1},d_{2}%
,d_{3}\right)  \right)   &  =\theta_{132}\circ\left(  \lambda_{\left(
d_{1},d_{2},d_{3}\right)  :D^{3}\left\{  \left(  2,3\right)  \right\}
}\left(  d_{1},d_{2},d_{3}\right)  \right) \\
\theta_{231}\circ\left(  \lambda_{\left(  d_{1},d_{2},d_{3}\right)
:D^{3}\left\{  \left(  2,3\right)  \right\}  }\left(  d_{1},d_{2}%
,d_{3}\right)  \right)   &  =\theta_{321}\circ\left(  \lambda_{\left(
d_{1},d_{2},d_{3}\right)  :D^{3}\left\{  \left(  2,3\right)  \right\}
}\left(  d_{1},d_{2},d_{3}\right)  \right) \\
\theta_{231}\circ\left(  \lambda_{\left(  d_{1},d_{2},d_{3}\right)
:D^{3}\left\{  \left(  1,3\right)  \right\}  }\left(  d_{1},d_{2}%
,d_{3}\right)  \right)   &  =\theta_{213}\circ\left(  \lambda_{\left(
d_{1},d_{2},d_{3}\right)  :D^{3}\left\{  \left(  1,3\right)  \right\}
}\left(  d_{1},d_{2},d_{3}\right)  \right) \\
\theta_{312}\circ\left(  \lambda_{\left(  d_{1},d_{2},d_{3}\right)
:D^{3}\left\{  \left(  1,3\right)  \right\}  }\left(  d_{1},d_{2}%
,d_{3}\right)  \right)   &  =\theta_{132}\circ\left(  \lambda_{\left(
d_{1},d_{2},d_{3}\right)  :D^{3}\left\{  \left(  1,3\right)  \right\}
}\left(  d_{1},d_{2},d_{3}\right)  \right) \\
\theta_{312}\circ\left(  \lambda_{\left(  d_{1},d_{2},d_{3}\right)
:D^{3}\left\{  \left(  1,2\right)  \right\}  }\left(  d_{1},d_{2}%
,d_{3}\right)  \right)   &  =\theta_{321}\circ\left(  \lambda_{\left(
d_{1},d_{2},d_{3}\right)  :D^{3}\left\{  \left(  1,2\right)  \right\}
}\left(  d_{1},d_{2},d_{3}\right)  \right) \\
\theta_{123}\circ\left(  \lambda_{\left(  d_{1},d_{2},d_{3}\right)
:D^{3}\left\{  \left(  1,2\right)  \right\}  }\left(  d_{1},d_{2}%
,d_{3}\right)  \right)   &  =\theta_{213}\circ\left(  \lambda_{\left(
d_{1},d_{2},d_{3}\right)  :D^{3}\left\{  \left(  1,2\right)  \right\}
}\left(  d_{1},d_{2},d_{3}\right)  \right)
\end{align*}
there exists
\[
\mathfrak{k}_{\left(  \theta_{123},\theta_{132},\theta_{213},\theta
_{231},\theta_{312},\theta_{321}\right)  }:\mathcal{G}\rightarrow M
\]
such that
\begin{align*}
\mathfrak{k}_{\left(  \theta_{123},\theta_{132},\theta_{213},\theta
_{231},\theta_{312},\theta_{321}\right)  }\circ f_{123}  &  =\theta_{123}\\
\mathfrak{k}_{\left(  \theta_{123},\theta_{132},\theta_{213},\theta
_{231},\theta_{312},\theta_{321}\right)  }\circ f_{132}  &  =\theta_{132}\\
\mathfrak{k}_{\left(  \theta_{123},\theta_{132},\theta_{213},\theta
_{231},\theta_{312},\theta_{321}\right)  }\circ f_{213}  &  =\theta_{213}\\
\mathfrak{k}_{\left(  \theta_{123},\theta_{132},\theta_{213},\theta
_{231},\theta_{312},\theta_{321}\right)  }\circ f_{231}  &  =\theta_{231}\\
\mathfrak{k}_{\left(  \theta_{123},\theta_{132},\theta_{213},\theta
_{231},\theta_{312},\theta_{321}\right)  }\circ f_{312}  &  =\theta_{312}\\
\mathfrak{k}_{\left(  \theta_{123},\theta_{132},\theta_{213},\theta
_{231},\theta_{312},\theta_{321}\right)  }\circ f_{312}  &  =\theta_{312}%
\end{align*}

\end{corollary}

\begin{theorem}
\label{t5.1}(\underline{The general Jacobi identity}) Let $M$\ be a
microlinear set and $\theta_{123},\theta_{132},\theta_{213},\theta
_{231},\theta_{312},\theta_{321}:D^{3}\rightarrow M$. If the identities
\begin{align*}
\theta_{123}\circ\left(  \lambda_{\left(  d_{1},d_{2},d_{3}\right)
:D^{3}\left\{  \left(  2,3\right)  \right\}  }\left(  d_{1},d_{2}%
,d_{3}\right)  \right)   &  =\theta_{132}\circ\left(  \lambda_{\left(
d_{1},d_{2},d_{3}\right)  :D^{3}\left\{  \left(  2,3\right)  \right\}
}\left(  d_{1},d_{2},d_{3}\right)  \right) \\
\theta_{231}\circ\left(  \lambda_{\left(  d_{1},d_{2},d_{3}\right)
:D^{3}\left\{  \left(  2,3\right)  \right\}  }\left(  d_{1},d_{2}%
,d_{3}\right)  \right)   &  =\theta_{321}\circ\left(  \lambda_{\left(
d_{1},d_{2},d_{3}\right)  :D^{3}\left\{  \left(  2,3\right)  \right\}
}\left(  d_{1},d_{2},d_{3}\right)  \right) \\
\theta_{231}\circ\left(  \lambda_{\left(  d_{1},d_{2},d_{3}\right)
:D^{3}\left\{  \left(  1,3\right)  \right\}  }\left(  d_{1},d_{2}%
,d_{3}\right)  \right)   &  =\theta_{213}\circ\left(  \lambda_{\left(
d_{1},d_{2},d_{3}\right)  :D^{3}\left\{  \left(  1,3\right)  \right\}
}\left(  d_{1},d_{2},d_{3}\right)  \right) \\
\theta_{312}\circ\left(  \lambda_{\left(  d_{1},d_{2},d_{3}\right)
:D^{3}\left\{  \left(  1,3\right)  \right\}  }\left(  d_{1},d_{2}%
,d_{3}\right)  \right)   &  =\theta_{132}\circ\left(  \lambda_{\left(
d_{1},d_{2},d_{3}\right)  :D^{3}\left\{  \left(  1,3\right)  \right\}
}\left(  d_{1},d_{2},d_{3}\right)  \right) \\
\theta_{312}\circ\left(  \lambda_{\left(  d_{1},d_{2},d_{3}\right)
:D^{3}\left\{  \left(  1,2\right)  \right\}  }\left(  d_{1},d_{2}%
,d_{3}\right)  \right)   &  =\theta_{321}\circ\left(  \lambda_{\left(
d_{1},d_{2},d_{3}\right)  :D^{3}\left\{  \left(  1,2\right)  \right\}
}\left(  d_{1},d_{2},d_{3}\right)  \right) \\
\theta_{123}\circ\left(  \lambda_{\left(  d_{1},d_{2},d_{3}\right)
:D^{3}\left\{  \left(  1,2\right)  \right\}  }\left(  d_{1},d_{2}%
,d_{3}\right)  \right)   &  =\theta_{213}\circ\left(  \lambda_{\left(
d_{1},d_{2},d_{3}\right)  :D^{3}\left\{  \left(  1,2\right)  \right\}
}\left(  d_{1},d_{2},d_{3}\right)  \right)
\end{align*}
obtain so that the six relative strong differences
\begin{align*}
&  \theta_{123}\underset{1}{\overset{\cdot}{-}}\theta_{132}\\
&  \theta_{231}\underset{1}{\overset{\cdot}{-}}\theta_{321}\\
&  \theta_{231}\underset{2}{\overset{\cdot}{-}}\theta_{213}\\
&  \theta_{312}\underset{2}{\overset{\cdot}{-}}\theta_{132}\\
&  \theta_{312}\underset{3}{\overset{\cdot}{-}}\theta_{321}\\
&  \theta_{123}\underset{3}{\overset{\cdot}{-}}\theta_{213}%
\end{align*}
are to be defined, then the identities
\begin{align}
\theta_{123}\circ\left(  \lambda_{\left(  d_{1},d_{2},d_{3}\right)
:D^{3}\left\{  \left(  1,2\right)  ,\left(  1,3\right)  \right\}  }\left(
d_{1},d_{2},d_{3}\right)  \right)   &  =\theta_{231}\circ\left(
\lambda_{\left(  d_{1},d_{2},d_{3}\right)  :D^{3}\left\{  \left(  1,2\right)
,\left(  1,3\right)  \right\}  }\left(  d_{1},d_{2},d_{3}\right)  \right)
\label{t5.1.1}\\
\theta_{132}\circ\left(  \lambda_{\left(  d_{1},d_{2},d_{3}\right)
:D^{3}\left\{  \left(  1,2\right)  ,\left(  1,3\right)  \right\}  }\left(
d_{1},d_{2},d_{3}\right)  \right)   &  =\theta_{321}\circ\left(
\lambda_{\left(  d_{1},d_{2},d_{3}\right)  :D^{3}\left\{  \left(  1,2\right)
,\left(  1,3\right)  \right\}  }\left(  d_{1},d_{2},d_{3}\right)  \right)
\label{t5.1.2}\\
\theta_{231}\circ\left(  \lambda_{\left(  d_{1},d_{2},d_{3}\right)
:D^{3}\left\{  \left(  1,2\right)  ,\left(  2,3\right)  \right\}  }\left(
d_{1},d_{2},d_{3}\right)  \right)   &  =\theta_{312}\circ\left(
\lambda_{\left(  d_{1},d_{2},d_{3}\right)  :D^{3}\left\{  \left(  1,2\right)
,\left(  2,3\right)  \right\}  }\left(  d_{1},d_{2},d_{3}\right)  \right)
\label{t5.1.3}\\
\theta_{213}\circ\left(  \lambda_{\left(  d_{1},d_{2},d_{3}\right)
:D^{3}\left\{  \left(  1,2\right)  ,\left(  2,3\right)  \right\}  }\left(
d_{1},d_{2},d_{3}\right)  \right)   &  =\theta_{132}\circ\left(
\lambda_{\left(  d_{1},d_{2},d_{3}\right)  :D^{3}\left\{  \left(  1,2\right)
,\left(  2,3\right)  \right\}  }\left(  d_{1},d_{2},d_{3}\right)  \right)
\label{t5.1.4}\\
\theta_{312}\circ\left(  \lambda_{\left(  d_{1},d_{2},d_{3}\right)
:D^{3}\left\{  \left(  1,3\right)  ,\left(  2,3\right)  \right\}  }\left(
d_{1},d_{2},d_{3}\right)  \right)   &  =\theta_{123}\circ\left(
\lambda_{\left(  d_{1},d_{2},d_{3}\right)  :D^{3}\left\{  \left(  1,3\right)
,\left(  2,3\right)  \right\}  }\left(  d_{1},d_{2},d_{3}\right)  \right)
\label{t5.1.5}\\
\theta_{321}\circ\left(  \lambda_{\left(  d_{1},d_{2},d_{3}\right)
:D^{3}\left\{  \left(  1,3\right)  ,\left(  2,3\right)  \right\}  }\left(
d_{1},d_{2},d_{3}\right)  \right)   &  =\theta_{213}\circ\left(
\lambda_{\left(  d_{1},d_{2},d_{3}\right)  :D^{3}\left\{  \left(  1,3\right)
,\left(  2,3\right)  \right\}  }\left(  d_{1},d_{2},d_{3}\right)  \right)
\label{t5.1.6}%
\end{align}
obtain so that the three strong differences
\begin{align*}
&  \left(  \theta_{123}\underset{1}{\overset{\cdot}{-}}\theta_{132}\right)
\overset{\cdot}{-}\left(  \theta_{231}\underset{1}{\overset{\cdot}{-}}%
\theta_{321}\right) \\
&  \left(  \theta_{231}\underset{2}{\overset{\cdot}{-}}\theta_{213}\right)
\overset{\cdot}{-}\left(  \theta_{312}\underset{2}{\overset{\cdot}{-}}%
\theta_{132}\right) \\
&  \left(  \theta_{312}\underset{3}{\overset{\cdot}{-}}\theta_{321}\right)
\overset{\cdot}{-}\left(  \theta_{123}\underset{3}{\overset{\cdot}{-}}%
\theta_{213}\right)
\end{align*}
are to be defined (by dint of Proposition \ref{p5.5}) and to sum up only to
vanish, namely,
\begin{align}
&  \left(  \left(  \theta_{123}\underset{1}{\overset{\cdot}{-}}\theta
_{132}\right)  \overset{\cdot}{-}\left(  \theta_{231}\underset{1}%
{\overset{\cdot}{-}}\theta_{321}\right)  \right)  +\left(  \left(
\theta_{231}\underset{2}{\overset{\cdot}{-}}\theta_{213}\right)
\overset{\cdot}{-}\left(  \theta_{312}\underset{2}{\overset{\cdot}{-}}%
\theta_{132}\right)  \right)  +\nonumber\\
&  \left(  \left(  \theta_{312}\underset{3}{\overset{\cdot}{-}}\theta
_{321}\right)  \overset{\cdot}{-}\left(  \theta_{123}\underset{3}%
{\overset{\cdot}{-}}\theta_{213}\right)  \right) \nonumber\\
&  =0 \label{t5.1.7}%
\end{align}

\end{theorem}

\begin{proof}
The proof is divided into the proof of (\ref{t5.1.1})-(\ref{t5.1.6}) and that
of (\ref{t5.1.7}).

\begin{enumerate}
\item Since the identity
\[
\theta_{231}\circ\left(  \lambda_{\left(  d_{1},d_{2},d_{3}\right)
:D^{3}\left\{  \left(  1,3\right)  \right\}  }\left(  d_{1},d_{2}%
,d_{3}\right)  \right)  =\theta_{213}\circ\left(  \lambda_{\left(  d_{1}%
,d_{2},d_{3}\right)  :D^{3}\left\{  \left(  1,3\right)  \right\}  }\left(
d_{1},d_{2},d_{3}\right)  \right)
\]
obtains by assumption, we have
\[
\theta_{231}\circ\left(  \lambda_{\left(  d_{1},d_{2},d_{3}\right)
:D^{3}\left\{  \left(  1,2\right)  ,\left(  1,3\right)  \right\}  }\left(
d_{1},d_{2},d_{3}\right)  \right)  =\theta_{213}\circ\left(  \lambda_{\left(
d_{1},d_{2},d_{3}\right)  :D^{3}\left\{  \left(  1,2\right)  ,\left(
1,3\right)  \right\}  }\left(  d_{1},d_{2},d_{3}\right)  \right)
\]
Since the identity
\[
\theta_{123}\circ\left(  \lambda_{\left(  d_{1},d_{2},d_{3}\right)
:D^{3}\left\{  \left(  1,2\right)  \right\}  }\left(  d_{1},d_{2}%
,d_{3}\right)  \right)  =\theta_{213}\circ\left(  \lambda_{\left(  d_{1}%
,d_{2},d_{3}\right)  :D^{3}\left\{  \left(  1,2\right)  \right\}  }\left(
d_{1},d_{2},d_{3}\right)  \right)
\]
obtains by assumption, we have
\[
\theta_{123}\circ\left(  \lambda_{\left(  d_{1},d_{2},d_{3}\right)
:D^{3}\left\{  \left(  1,2\right)  ,\left(  1,3\right)  \right\}  }\left(
d_{1},d_{2},d_{3}\right)  \right)  =\theta_{213}\circ\left(  \lambda_{\left(
d_{1},d_{2},d_{3}\right)  :D^{3}\left\{  \left(  1,2\right)  ,\left(
1,3\right)  \right\}  }\left(  d_{1},d_{2},d_{3}\right)  \right)
\]
Therefore we have
\[
\theta_{123}\circ\left(  \lambda_{\left(  d_{1},d_{2},d_{3}\right)
:D^{3}\left\{  \left(  1,2\right)  ,\left(  1,3\right)  \right\}  }\left(
d_{1},d_{2},d_{3}\right)  \right)  =\theta_{231}\circ\left(  \lambda_{\left(
d_{1},d_{2},d_{3}\right)  :D^{3}\left\{  \left(  1,2\right)  ,\left(
1,3\right)  \right\}  }\left(  d_{1},d_{2},d_{3}\right)  \right)
\]
which is no other than (\ref{t5.1.1}). The remaining five identities
(\ref{t5.1.2})-(\ref{t5.1.6}) can be dealt with by the same token.

\item Since the identities
\begin{align*}
&  \lambda_{\left(  d_{1},d_{2},d_{3}\right)  :D^{3}}\theta_{123}\left(
d_{1},d_{2},d_{3}\right) \\
&  =\lambda_{\left(  d_{1},d_{2},d_{3}\right)  :D^{3}}\mathfrak{k}_{\left(
\theta_{123},\theta_{132},\theta_{213},\theta_{231},\theta_{312},\theta
_{321}\right)  }\left(  d_{1},d_{2},d_{3},0,0,0,0,0\right) \\
&  \lambda_{\left(  d_{1},d_{2},d_{3}\right)  :D^{3}}\theta_{132}\left(
d_{1},d_{2},d_{3}\right) \\
&  =\lambda_{\left(  d_{1},d_{2},d_{3}\right)  :D^{3}}\mathfrak{k}_{\left(
\theta_{123},\theta_{132},\theta_{213},\theta_{231},\theta_{312},\theta
_{321}\right)  }\left(  d_{1},d_{2},d_{3},d_{2}d_{3},0,0,0,0\right) \\
&  \lambda_{\left(  d_{1},d_{2},d_{3}\right)  :D^{3}}\theta_{213}\left(
d_{1},d_{2},d_{3}\right) \\
&  =\lambda_{\left(  d_{1},d_{2},d_{3}\right)  :D^{3}}\mathfrak{k}_{\left(
\theta_{123},\theta_{132},\theta_{213},\theta_{231},\theta_{312},\theta
_{321}\right)  }\left(  d_{1},d_{2},d_{3},0,0,d_{1}d_{2},0,0\right) \\
&  \lambda_{\left(  d_{1},d_{2},d_{3}\right)  :D^{3}}\theta_{231}\left(
d_{1},d_{2},d_{3}\right) \\
&  =\lambda_{\left(  d_{1},d_{2},d_{3}\right)  :D^{3}}\mathfrak{k}_{\left(
\theta_{123},\theta_{132},\theta_{213},\theta_{231},\theta_{312},\theta
_{321}\right)  }\left(  d_{1},d_{2},d_{3},0,d_{1}d_{3},d_{1}d_{2},0,0\right)
\\
&  \lambda_{\left(  d_{1},d_{2},d_{3}\right)  :D^{3}}\theta_{312}\left(
d_{1},d_{2},d_{3}\right) \\
&  =\lambda_{\left(  d_{1},d_{2},d_{3}\right)  :D^{3}}\mathfrak{k}_{\left(
\theta_{123},\theta_{132},\theta_{213},\theta_{231},\theta_{312},\theta
_{321}\right)  }\left(  d_{1},d_{2},d_{3},d_{2}d_{3},d_{1}d_{3},0,d_{1}%
d_{2}d_{3},0\right) \\
&  \lambda_{\left(  d_{1},d_{2},d_{3}\right)  :D^{3}}\theta_{321}\left(
d_{1},d_{2},d_{3}\right) \\
&  =\lambda_{\left(  d_{1},d_{2},d_{3}\right)  :D^{3}}\mathfrak{k}_{\left(
\theta_{123},\theta_{132},\theta_{213},\theta_{231},\theta_{312},\theta
_{321}\right)  }\left(  d_{1},d_{2},d_{3},d_{2}d_{3},d_{1}d_{3},d_{1}%
d_{2},0,d_{1}d_{2}d_{3}\right)
\end{align*}
obtain so that the identities
\begin{align*}
&  \lambda_{\left(  d_{1},d_{2}\right)  :D^{2}}\left(  \theta_{123}%
\underset{1}{\overset{\cdot}{-}}\theta_{132}\right)  \left(  d_{1}%
,d_{2}\right) \\
&  =\lambda_{\left(  d_{1},d_{2}\right)  :D^{2}}\mathfrak{k}_{\left(
\theta_{123},\theta_{132},\theta_{213},\theta_{231},\theta_{312},\theta
_{321}\right)  }\left(  d_{1},0,0,-d_{2},0,0,0,0\right) \\
&  \lambda_{\left(  d_{1},d_{2}\right)  :D^{2}}\left(  \theta_{231}%
\underset{1}{\overset{\cdot}{-}}\theta_{321}\right)  \left(  d_{1}%
,d_{2}\right) \\
&  =\lambda_{\left(  d_{1},d_{2}\right)  :D^{2}}\mathfrak{k}_{\left(
\theta_{123},\theta_{132},\theta_{213},\theta_{231},\theta_{312},\theta
_{321}\right)  }\left(  d_{1},0,0,-d_{2},0,0,0,-d_{1}d_{2}\right) \\
&  \lambda_{\left(  d_{1},d_{2}\right)  :D^{2}}\left(  \theta_{231}%
\underset{2}{\overset{\cdot}{-}}\theta_{213}\right)  \left(  d_{1}%
,d_{2}\right) \\
&  =\lambda_{\left(  d_{1},d_{2}\right)  :D^{2}}\mathfrak{k}_{\left(
\theta_{123},\theta_{132},\theta_{213},\theta_{231},\theta_{312},\theta
_{321}\right)  }\left(  0.d_{1},0,d_{2},0,0,0,0\right) \\
&  \lambda_{\left(  d_{1},d_{2}\right)  :D^{2}}\left(  \theta_{312}%
\underset{2}{\overset{\cdot}{-}}\theta_{132}\right)  \left(  d_{1}%
,d_{2}\right) \\
&  =\lambda_{\left(  d_{1},d_{2}\right)  :D^{2}}\mathfrak{k}_{\left(
\theta_{123},\theta_{132},\theta_{213},\theta_{231},\theta_{312},\theta
_{321}\right)  }\left(  0,d_{1},0,0,d_{2},0,d_{1}d_{2},0\right) \\
&  \lambda_{\left(  d_{1},d_{2}\right)  :D^{2}}\left(  \theta_{312}%
\underset{3}{\overset{\cdot}{-}}\theta_{321}\right)  \left(  d_{1}%
,d_{2}\right) \\
&  =\lambda_{\left(  d_{1},d_{2}\right)  :D^{2}}\mathfrak{k}_{\left(
\theta_{123},\theta_{132},\theta_{213},\theta_{231},\theta_{312},\theta
_{321}\right)  }\left(  0,0,d_{1},0,0,-d_{2},d_{1}d_{2},-d_{1}d_{2}\right) \\
&  \lambda_{\left(  d_{1},d_{2}\right)  :D^{2}}\left(  \theta_{123}%
\underset{3}{\overset{\cdot}{-}}\theta_{213}\right)  \left(  d_{1}%
,d_{2}\right) \\
&  =\lambda_{\left(  d_{1},d_{2}\right)  :D^{2}}\mathfrak{k}_{\left(
\theta_{123},\theta_{132},\theta_{213},\theta_{231},\theta_{312},\theta
_{321}\right)  }\left(  0,0,d_{1},0,0,-d_{2},0,0\right)
\end{align*}
take place, we have
\begin{align*}
&  \lambda_{d:D}\left(  \left(  \theta_{123}\underset{1}{\overset{\cdot}{-}%
}\theta_{132}\right)  \overset{\cdot}{-}\left(  \theta_{231}\underset
{1}{\overset{\cdot}{-}}\theta_{321}\right)  \right)  \left(  d\right) \\
&  =\lambda_{d:D}\mathfrak{k}_{\left(  \theta_{123},\theta_{132},\theta
_{213},\theta_{231},\theta_{312},\theta_{321}\right)  }\left(
0,0,0,0,0,0,0,d\right) \\
&  \lambda_{d:D}\left(  \left(  \theta_{231}\underset{2}{\overset{\cdot}{-}%
}\theta_{213}\right)  \overset{\cdot}{-}\left(  \theta_{312}\underset
{2}{\overset{\cdot}{-}}\theta_{132}\right)  \right)  \left(  d\right) \\
&  =\lambda_{d:D}\mathfrak{k}_{\left(  \theta_{123},\theta_{132},\theta
_{213},\theta_{231},\theta_{312},\theta_{321}\right)  }\left(
0,0,0,0,0,0,-d,0\right) \\
&  \lambda_{d:D}\left(  \left(  \theta_{312}\underset{3}{\overset{\cdot}{-}%
}\theta_{321}\right)  \overset{\cdot}{-}\left(  \theta_{123}\underset
{3}{\overset{\cdot}{-}}\theta_{213}\right)  \right)  \left(  d\right) \\
&  =\lambda_{d:D}\mathfrak{k}_{\left(  \theta_{123},\theta_{132},\theta
_{213},\theta_{231},\theta_{312},\theta_{321}\right)  }\left(
0,0,0,0,0,0,d,-d\right)
\end{align*}
This completes the proof.
\end{enumerate}
\end{proof}

\section{\label{s6}Vector Fields}

The trinity of the three notions of vector fields in synthetic differential
geometry, namely the identification of sections of tangent vector bundles,
infinitesimal flows and infinitesimal transformations discussed in
\S 3.2.1\ of \cite{la}, remains valid in the following sense.

\begin{theorem}
\label{t6.1}Let $M$\ be a microlinear type. The following three types are
mutually equivalent.

\begin{itemize}
\item the type of sections of the fibration $\lambda_{x:M}\mathbf{T}%
_{x}M:M\rightarrow\mathcal{U}$, namely,%
\[%
{\displaystyle\prod\limits_{x:M}}
\mathbf{T}_{x}M
\]

\item the type of infinitesimal flows on $\left\Vert M\right\Vert _{0}$,
namely, mappings $f:D\times\left\Vert M\right\Vert _{0}\rightarrow\left\Vert
M\right\Vert _{0}$ in accordance with%
\[%
{\displaystyle\prod\limits_{x:\left\Vert M\right\Vert _{0}}}
f(0,x)=x
\]

\item the type of infinitesimal transformations of $\left\Vert M\right\Vert
_{0}$, namely, mappings $X:D\rightarrow\left\Vert M\right\Vert _{0}%
\rightarrow\left\Vert M\right\Vert _{0}$ in accordance with%
\[
X_{0}=\mathrm{id}_{\left\Vert M\right\Vert _{0}}%
\]
where we prefer to write $X_{d}$ in place of $X\left(  d\right)  $ as in
\cite{la}.
\end{itemize}
\end{theorem}

\begin{proof}
A section of the dependent type family $\lambda_{x:M}\mathbf{T}_{x}M$\ can be
identified with a mapping%
\[
\widetilde{f}:M\rightarrow D\rightarrow\left\Vert M\right\Vert _{0}%
\]
in accordance with%
\[%
{\displaystyle\prod\limits_{x:M}}
\widetilde{f}(x,0)=\left\vert x\right\vert _{0}%
\]
which can naturally be identified with a mapping%
\[
\widehat{f}:D\rightarrow M\rightarrow\left\Vert M\right\Vert _{0}%
\]
in accordance with%
\[%
{\displaystyle\prod\limits_{x:M}}
\widehat{f}(0,x)=\left\vert x\right\vert _{0}%
\]
Since%
\[
M\rightarrow\left\Vert M\right\Vert _{0}\backsimeq\left\Vert M\right\Vert
_{0}\rightarrow\left\Vert M\right\Vert _{0}%
\]
obtains naturally, $\widehat{f}$ can be identified with a mapping%
\[
f:D\rightarrow\left\Vert M\right\Vert _{0}\rightarrow\left\Vert M\right\Vert
_{0}%
\]
in accordance with%
\[%
{\displaystyle\prod\limits_{x:\left\Vert M\right\Vert _{0}}}
f(0,x)=x
\]
This has established the equivalence between the first and the second. The
equivalence between the second and the third can be established more directly,
which is safely left to the reader.
\end{proof}

\begin{notation}
Given a microlinear type $M$, the notaion $\mathfrak{X}\left(  M\right)
$\ denotes one of the equivalent three types in the above theorem, but it
usually means the third one in the rest of this paper unless specified otherwise.
\end{notation}

\begin{proposition}
\label{p6.1}(cf. Proposition 3 of \S 3.2 in \cite{la}) Let $M$\ be a
microlinear type. Let $X:\mathfrak{X}\left(  M\right)  $. Then we have%
\[%
{\displaystyle\prod\limits_{\left(  d_{1},d_{2}\right)  :D\left(  2\right)  }}
X_{d_{1}}\circ X_{d_{2}}=X_{d_{1}+d_{2}}%
\]

\end{proposition}

\begin{proof}
It is easy to see that%
\begin{align*}
\left(  \lambda_{\left(  d_{1},d_{2}\right)  :D\left(  2\right)  }X_{d_{1}%
}\circ X_{d_{2}}\right)  \circ\left(  \lambda_{d:D}\left(  d,0\right)
\right)   &  =\left(  \lambda_{\left(  d_{1},d_{2}\right)  :D\left(  2\right)
}X_{d_{1}+d_{2}}\right)  \circ\left(  \lambda_{d:D}\left(  d,0\right)  \right)
\\
\left(  \lambda_{\left(  d_{1},d_{2}\right)  :D\left(  2\right)  }X_{d_{1}%
}\circ X_{d_{2}}\right)  \circ\left(  \lambda_{d:D}\left(  0,d\right)
\right)   &  =\left(  \lambda_{\left(  d_{1},d_{2}\right)  :D\left(  2\right)
}X_{d_{1}+d_{2}}\right)  \circ\left(  \lambda_{d:D}\left(  0,d\right)
\right)
\end{align*}
so that the desired result follows by dint of Corollary \ref{cl4.1}.
\end{proof}

\begin{proposition}
\label{p6.2}(cf. Proposition 6 of \S 3.2 in \cite{la}) Let $M$\ be a
microlinear type. Let $X,Y:\mathfrak{X}\left(  M\right)  $. Then we have%
\begin{align*}%
{\displaystyle\prod\limits_{d:D}}
X_{d}\circ Y_{d}  &  =\left(  X+Y\right)  _{d}\\%
{\displaystyle\prod\limits_{d:D}}
Y_{d}\circ X_{d}  &  =\left(  X+Y\right)  _{d}%
\end{align*}

\end{proposition}

\begin{proof}
It is easy to see that%
\begin{align*}
\left(  \lambda_{\left(  d_{1},d_{2}\right)  :D\left(  2\right)  }X_{d_{1}%
}\circ Y_{d_{2}}\right)  \circ\left(  \lambda_{d:D}\left(  d,0\right)
\right)   &  =\lambda_{d:D}X_{d}\\
\left(  \lambda_{\left(  d_{1},d_{2}\right)  :D\left(  2\right)  }X_{d_{1}%
}\circ Y_{d_{2}}\right)  \circ\left(  \lambda_{d:D}\left(  0,d\right)
\right)   &  =\lambda_{d:D}Y_{d}%
\end{align*}
so that the first desired result follows by dint of Corollary \ref{cl4.1}.%
\begin{align*}
\left(  \lambda_{\left(  d_{1},d_{2}\right)  :D\left(  2\right)  }Y_{d_{2}%
}\circ X_{d_{1}}\right)  \circ\left(  \lambda_{d:D}\left(  d,0\right)
\right)   &  =\lambda_{d:D}X_{d}\\
\left(  \lambda_{\left(  d_{1},d_{2}\right)  :D\left(  2\right)  }Y_{d_{2}%
}\circ X_{d_{1}}\right)  \circ\left(  \lambda_{d:D}\left(  0,d\right)
\right)   &  =\lambda_{d:D}Y_{d}%
\end{align*}
so that the second desired result follows by dint of Corollary \ref{cl4.1}.
\end{proof}

\begin{definition}
Let $M$ be a microlinear set. Given
\begin{align*}
\theta_{1}  &  :D^{n_{1}}\rightarrow M\rightarrow M\\
\theta_{2}  &  :D^{n_{2}}\rightarrow M\rightarrow M
\end{align*}
we define
\[
\theta_{1}\ast\theta_{2}:D^{n_{1}+n_{2}}\rightarrow M\rightarrow M
\]
to be
\[
\theta_{1}\ast\theta_{2}:\equiv\lambda_{\left(  d_{1},...,d_{n_{1}+n_{2}%
}\right)  :D^{n_{1}+n_{2}}}\theta_{1}\left(  d_{n_{2}+1},...,d_{n_{1}+n_{2}%
}\right)  \circ\theta_{2}\left(  d_{1},...,d_{n_{2}}\right)
\]

\end{definition}

It is easy to see that

\begin{lemma}
\label{l6.2}Let $M$ be a microlinear type. Given
\begin{align*}
\theta_{1}  &  :D^{n_{1}}\rightarrow M\rightarrow M\\
\theta_{2}  &  :D^{n_{2}}\rightarrow M\rightarrow M\\
\theta_{3}  &  :D^{n_{3}}\rightarrow M\rightarrow M
\end{align*}
we have
\[
\left(  \theta_{1}\ast\theta_{2}\right)  \ast\theta_{3}=\theta_{1}\ast\left(
\theta_{2}\ast\theta_{3}\right)
\]

\end{lemma}

\begin{remark}
Therefore, when various $\theta_{i}:D^{n_{i}}\rightarrow M\rightarrow
M$\ $\left(  1\leq i\leq m\right)  $ are concatenated by $\ast$, we can omit
parentheses, so that we can write
\[
\theta_{1}\ast...\ast\theta_{m}:D^{n_{1}+...+n_{m}}\rightarrow M\rightarrow M
\]

\end{remark}

\begin{definition}
Let $M$\ be a microlinear type. Given $X,Y:\mathfrak{X}\left(  M\right)  $, we
have
\begin{align*}
&  \left(  Y\ast X\right)  \circ\left(  \lambda_{\left(  d_{1},d_{2}\right)
:D^{2}\left\{  \left(  1,2\right)  \right\}  }\left(  d_{1},d_{2}\right)
\right) \\
&  =\left(  X\ast Y\right)  \circ\left(  \lambda_{\left(  d_{1},d_{2}\right)
:D^{2}}\left(  d_{2},d_{1}\right)  \right)  \circ\left(  \lambda_{\left(
d_{1},d_{2}\right)  :D^{2}\left\{  \left(  1,2\right)  \right\}  }\left(
d_{1},d_{2}\right)  \right)
\end{align*}
so that we can define%
\[
\left[  X,Y\right]  :\equiv\left(  Y\ast X\right)  \overset{\cdot}{-}\left(
X\ast Y\right)  \circ\left(  \lambda_{\left(  d_{1},d_{2}\right)  :D^{2}%
}\left(  d_{2},d_{1}\right)  \right)
\]

\end{definition}

\begin{lemma}
\label{l6.3}Let $M$ be a microlinear type. Given $X_{1},X_{2},X_{3}%
:\mathfrak{X}\left(  M\right)  $, we define
\begin{align*}
\theta_{123}  &  :\equiv X_{3}\ast X_{2}\ast X_{1}\\
\theta_{132}  &  :\equiv\left(  X_{2}\ast X_{3}\ast X_{1}\right)  \circ\left(
\lambda_{\left(  d_{1},d_{2},d_{3}\right)  :D^{3}}\left(  d_{1},d_{3}%
,d_{2}\right)  \right) \\
\theta_{231}  &  :\equiv\left(  X_{1}\ast X_{3}\ast X_{2}\right)  \circ\left(
\lambda_{\left(  d_{1},d_{2},d_{3}\right)  :D^{3}}\left(  d_{2},d_{3}%
,d_{1}\right)  \right) \\
\theta_{321}  &  :\equiv\left(  X_{1}\ast X_{2}\ast X_{3}\right)  \circ\left(
\lambda_{\left(  d_{1},d_{2},d_{3}\right)  :D^{3}}\left(  d_{3},d_{2}%
,d_{1}\right)  \right)
\end{align*}
Then we have
\begin{align*}
\theta_{123}\circ\left(  \lambda_{\left(  d_{1},d_{2},d_{3}\right)
:D^{3}\left\{  \left(  2,3\right)  \right\}  }\left(  d_{1},d_{2}%
,d_{3}\right)  \right)   &  =\theta_{132}\circ\left(  \lambda_{\left(
d_{1},d_{2},d_{3}\right)  :D^{3}\left\{  \left(  2,3\right)  \right\}
}\left(  d_{1},d_{2},d_{3}\right)  \right) \\
\theta_{231}\circ\left(  \lambda_{\left(  d_{1},d_{2},d_{3}\right)
:D^{3}\left\{  \left(  2,3\right)  \right\}  }\left(  d_{1},d_{2}%
,d_{3}\right)  \right)   &  =\theta_{321}\circ\left(  \lambda_{\left(
d_{1},d_{2},d_{3}\right)  :D^{3}\left\{  \left(  2,3\right)  \right\}
}\left(  d_{1},d_{2},d_{3}\right)  \right) \\
\theta_{123}\circ\left(  \lambda_{\left(  d_{1},d_{2},d_{3}\right)
:D^{3}\left\{  \left(  1,2\right)  ,\left(  1,3\right)  \right\}  }\left(
d_{1},d_{2},d_{3}\right)  \right)   &  =\theta_{231}\circ\left(
\lambda_{\left(  d_{1},d_{2},d_{3}\right)  :D^{3}\left\{  \left(  1,2\right)
,\left(  1,3\right)  \right\}  }\left(  d_{1},d_{2},d_{3}\right)  \right) \\
\theta_{132}\circ\left(  \lambda_{\left(  d_{1},d_{2},d_{3}\right)
:D^{3}\left\{  \left(  1,2\right)  ,\left(  1,3\right)  \right\}  }\left(
d_{1},d_{2},d_{3}\right)  \right)   &  =\theta_{321}\circ\left(
\lambda_{\left(  d_{1},d_{2},d_{3}\right)  :D^{3}\left\{  \left(  1,2\right)
,\left(  1,3\right)  \right\}  }\left(  d_{1},d_{2},d_{3}\right)  \right) \\
\left[  X_{2},X_{3}\right]  \ast X_{1}  &  =\theta_{123}\underset{1}%
{\overset{\cdot}{-}}\theta_{132}\\
\left(  X_{1}\ast\left[  X_{2},X_{3}\right]  \right)  \circ\left(
\lambda_{\left(  d_{1},d_{2}\right)  :D^{2}}\left(  d_{2},d_{1}\right)
\right)   &  =\theta_{231}\underset{1}{\overset{\cdot}{-}}\theta_{321}\\
\left[  X_{1},\left[  X_{2},X_{3}\right]  \right]   &  =\left(  \theta
_{123}\underset{1}{\overset{\cdot}{-}}\theta_{132}\right)  \overset{\cdot}%
{-}\left(  \theta_{231}\underset{1}{\overset{\cdot}{-}}\theta_{321}\right)
\end{align*}

\end{lemma}

\begin{lemma}
\label{l6.4}Let $M$ be a microlinear type. Given $X_{1},X_{2},X_{3}%
:\mathfrak{X}\left(  M\right)  $, we define
\begin{align*}
\theta_{123}  &  :\equiv X_{3}\ast X_{2}\ast X_{1}\\
\theta_{132}  &  :\equiv\left(  X_{2}\ast X_{3}\ast X_{1}\right)  \circ\left(
\lambda_{\left(  d_{1},d_{2},d_{3}\right)  :D^{3}}\left(  d_{1},d_{3}%
,d_{2}\right)  \right) \\
\theta_{213}  &  :\equiv\left(  X_{3}\ast X_{1}\ast X_{2}\right)  \circ\left(
\lambda_{\left(  d_{1},d_{2},d_{3}\right)  :D^{3}}\left(  d_{2},d_{1}%
,d_{3}\right)  \right) \\
\theta_{231}  &  :\equiv\left(  X_{1}\ast X_{3}\ast X_{2}\right)  \circ\left(
\lambda_{\left(  d_{1},d_{2},d_{3}\right)  :D^{3}}\left(  d_{2},d_{3}%
,d_{1}\right)  \right) \\
\theta_{312}  &  :\equiv\left(  X_{2}\ast X_{1}\ast X_{3}\right)  \circ\left(
\lambda_{\left(  d_{1},d_{2},d_{3}\right)  :D^{3}}\left(  d_{3},d_{1}%
,d_{2}\right)  \right) \\
\theta_{321}  &  :\equiv\left(  X_{1}\ast X_{2}\ast X_{3}\right)  \circ\left(
\lambda_{\left(  d_{1},d_{2},d_{3}\right)  :D^{3}}\left(  d_{3},d_{2}%
,d_{1}\right)  \right)
\end{align*}
Then we have
\begin{align*}
\theta_{123}\circ\left(  \lambda_{\left(  d_{1},d_{2},d_{3}\right)
:D^{3}\left\{  \left(  2,3\right)  \right\}  }\left(  d_{1},d_{2}%
,d_{3}\right)  \right)   &  =\theta_{132}\circ\left(  \lambda_{\left(
d_{1},d_{2},d_{3}\right)  :D^{3}\left\{  \left(  2,3\right)  \right\}
}\left(  d_{1},d_{2},d_{3}\right)  \right) \\
\theta_{231}\circ\left(  \lambda_{\left(  d_{1},d_{2},d_{3}\right)
:D^{3}\left\{  \left(  2,3\right)  \right\}  }\left(  d_{1},d_{2}%
,d_{3}\right)  \right)   &  =\theta_{321}\circ\left(  \lambda_{\left(
d_{1},d_{2},d_{3}\right)  :D^{3}\left\{  \left(  2,3\right)  \right\}
}\left(  d_{1},d_{2},d_{3}\right)  \right) \\
\theta_{231}\circ\left(  \lambda_{\left(  d_{1},d_{2},d_{3}\right)
:D^{3}\left\{  \left(  1,3\right)  \right\}  }\left(  d_{1},d_{2}%
,d_{3}\right)  \right)   &  =\theta_{213}\circ\left(  \lambda_{\left(
d_{1},d_{2},d_{3}\right)  :D^{3}\left\{  \left(  1,3\right)  \right\}
}\left(  d_{1},d_{2},d_{3}\right)  \right) \\
\theta_{312}\circ\left(  \lambda_{\left(  d_{1},d_{2},d_{3}\right)
:D^{3}\left\{  \left(  1,3\right)  \right\}  }\left(  d_{1},d_{2}%
,d_{3}\right)  \right)   &  =\theta_{132}\circ\left(  \lambda_{\left(
d_{1},d_{2},d_{3}\right)  :D^{3}\left\{  \left(  1,3\right)  \right\}
}\left(  d_{1},d_{2},d_{3}\right)  \right) \\
\theta_{312}\circ\left(  \lambda_{\left(  d_{1},d_{2},d_{3}\right)
:D^{3}\left\{  \left(  1,2\right)  \right\}  }\left(  d_{1},d_{2}%
,d_{3}\right)  \right)   &  =\theta_{321}\circ\left(  \lambda_{\left(
d_{1},d_{2},d_{3}\right)  :D^{3}\left\{  \left(  1,2\right)  \right\}
}\left(  d_{1},d_{2},d_{3}\right)  \right) \\
\theta_{123}\circ\left(  \lambda_{\left(  d_{1},d_{2},d_{3}\right)
:D^{3}\left\{  \left(  1,2\right)  \right\}  }\left(  d_{1},d_{2}%
,d_{3}\right)  \right)   &  =\theta_{213}\circ\left(  \lambda_{\left(
d_{1},d_{2},d_{3}\right)  :D^{3}\left\{  \left(  1,2\right)  \right\}
}\left(  d_{1},d_{2},d_{3}\right)  \right)
\end{align*}

\end{lemma}

\begin{theorem}
\label{t6.3}Let $M$ be a microlinear type. Given $\alpha:\mathbb{R}$\ and
$X_{1},X_{2},X_{3}:\mathfrak{X}\left(  M\right)  $, we have
\begin{align}
\left[  X_{1}+X_{2},X_{3}\right]   &  =\left[  X_{1},X_{3}\right]  +\left[
X_{2},X_{3}\right] \label{t6.3.1}\\
\left[  \alpha X_{1},X_{2}\right]   &  =\alpha\left[  X_{1},X_{2}\right]
\label{t6.3.2}\\
\left[  X_{1},X_{2}\right]  +\left[  X_{2},X_{1}\right]   &  =0 \label{t6.3.3}%
\\
\left[  X_{1},\left[  X_{2},X_{3}\right]  \right]  +\left[  X_{2},\left[
X_{3},X_{1}\right]  \right]  +\left[  X_{3},\left[  X_{1},X_{2}\right]
\right]   &  =0 \label{t6.3.4}%
\end{align}
In a word, the Lie bracket $\left[  \cdot,\cdot\right]  $\ is bilinear,
antisymmetric, and satisfies the Jacobi identity.
\end{theorem}

\begin{proof}
The property (\ref{t6.3.1}) follows from Proposition \ref{p5.2}. The property
(\ref{pt6.3.2}) follows from Proposition \ref{p5.3}, The property
(\ref{t6.3.3}) follows from Proposition \ref{p5.4}. The property
(\ref{t6.3.4}) follows from Theorem \ref{t5.1}.
\end{proof}

\end{document}